\RequirePackage{fix-cm}
\documentclass[smallextended]{svjour3}       % onecolumn (second format)
\smartqed  % flush right qed marks, e.g. at end of proof
\usepackage{}
\usepackage{graphicx}
\usepackage{multirow,bigstrut}
\usepackage{amsmath,amsfonts,amssymb}
\usepackage{mathrsfs}
\usepackage{color}
\usepackage{indentfirst}
\usepackage{algorithm,algpseudocode}
\usepackage{latexsym}
\usepackage{graphicx}
\usepackage{array,epsf,epsfig}
\usepackage{mathdots}
\usepackage{cite}
\usepackage{multirow}
\usepackage{epstopdf}
\usepackage{mathtools}
\usepackage{tensor}
\usepackage{textcomp}
\usepackage{booktabs}

\newcommand{\mO}{{\mathcal{O}}}

\newcommand{\beq}{\begin{equation}}
\newcommand{\eeq}{\end{equation}}
\newcommand{\bey}{\begin{eqnarray}}
\newcommand{\eey}{\end{eqnarray}}

\begin{document}

\title{Fast second-order evaluation for  variable-order Caputo fractional derivative with applications to fractional sub-diffusion equations
\thanks{This work is supported in part by research grants of the Science and Technology Development Fund, Macau SAR (file no. 0118/2018/A3), and University of Macau (file no. MYRG2018-00015-FST).}}

\author{Jia-li Zhang$^1$ \and Zhi-wei Fang$^2$ \and Hai-wei Sun$^3$
}


\institute{1 \at
              Department of Mathematics, University of Macau,
Macao\\
\email{zhangjl2628@163.com}\\
           2 \at
              Department of Mathematics, University of Macau,
Macao\\
\email{fzw913@yeah.net}\\
3 \at
              Corresponding Author. Department of Mathematics, University of Macau,
Macao\\
\email{HSun@um.edu.mo}
}

\date{Received: date / Accepted: date}

\maketitle

\begin{abstract}
In this paper, we propose a fast second-order approximation to the variable-order (VO) Caputo fractional derivative, which is developed based on $L2$-$1_\sigma$ formula and the exponential-sum-approximation technique. The fast evaluation method can achieve the second-order accuracy and further reduce the computational cost and the acting memory for the VO Caputo fractional derivative. This fast algorithm  is applied to construct a relevant fast temporal second-order and spatial fourth-order scheme ($FL2$-$1_{\sigma}$ scheme) for the multi-dimensional VO time-fractional sub-diffusion equations. Theoretically, $FL2$-$1_{\sigma}$ scheme is proved to fulfill the similar properties of the coefficients as those of the well-studied $L2$-$1_\sigma$ scheme. Therefore, $FL2$-$1_{\sigma}$ scheme is strictly proved to be unconditionally stable and convergent. A sharp decrease in the computational cost and the acting memory  is shown in the numerical examples to demonstrate the efficiency of the proposed method.

\end{abstract}

%\section{Introduction}
%\label{intro}
%Your text comes here. Separate text sections with
%\section{Section title}
%\label{sec:1}
%Text with citations \cite{RefB} and \cite{RefJ}.
%\subsection{Subsection title}
%\label{sec:2}
%as required. Don't forget to give each section
%and subsection a unique label (see Sect.~\ref{sec:1}).
%\paragraph{Paragraph headings} Use paragraph headings as needed.
%\begin{equation}
%\end{equation}
\section{Introduction}
Over the last few decades, the fractional calculus has gained much attention of both mathematics and physical science due to the non-locality of fractional operators. In fact the fractional calculus has been widely applied in various fields including the biology, the ecology, the diffusion, and the control system \cite{Hilfer-2000, Benson-2000, Kilbas-2006, Liu-2004, Podlubny-1999, Mainardi-2000, Raberto-2002, Oldham-1974}. Recently, more and more researchers revealed that many important dynamical problems exhibit the fractional order behavior that may vary with time, space, or some other factors, which leads to the concept of the \textit{variable-order} (VO) fractional operator; see \cite{Lorenzo-2002, Sun-2011, Sun-2019}. This fact indicates that the VO fractional calculus is an expected tool to provide an effective mathematical framework to characterize the complex problems in fields of science and engineering; for instances, anomalous diffusion \cite{Chechkin-2005, Sun-2009, Zhuang-2009, Chen-2010}, viscoelastic mechanics \cite{Coimbra-2003, Diazand-2009, Jia-2017, Sokolov-2005}, and petroleum engineering \cite{Obembe-2017}.

The VO fractional derivative can be regarded as an extension of the constant-order (CO) fractional derivative \cite{Sun-2015}. An interesting extension of the classical fractional calculus was proposed by Samko and Ross \cite{Samko-1993} in which they generalized the Riemann-Liouville and Marchaud fractional operators in the VO case. Later, Lorenzo and Hartley \cite{Lorenzo-2002} first introduced the concept of VO operators. According to the concept, the order of the operator is allowed to vary as a function of independent variables such as time and space. Afterwards,  various VO differential operators with specific meanings were defined. Coimbra \cite{Coimbra-2003} gave a novel definition for the VO differential operator by taking the Laplace-transform of the Caputo's definition of the fractional derivative. Soon, Coimbra, and Kobayashi \cite{Soon-2005} showed that Coimbra's definition was better suited for modeling physical problems due to its meaningful physical interpretations; see also \cite{Ramirez-2010}. Moreover, the Coimbra's VO fractional derivative could  be considered as the Caputo-type definition, which is defined as follows \cite{Coimbra-2003, Sun-2019}
\begin{align}\label{caputo-def}
\prescript{C}{0}{\mathcal{D}}^{\alpha(t)}_tu(t)\equiv\frac 1 {\Gamma(1-\alpha(t))}\int_0^t\frac{u'(\tau)}{(t-\tau)^{\alpha(t)}}\mathrm{d}\tau,
\end{align}
where $\alpha(t)\in[\underline{\alpha},\overline{\alpha}]\subset(0,1)$ is the VO function and $\Gamma(\cdot)$ denotes the Gamma function.

Since the problems described by the VO fractional operator are difficult to handle analytically, possible numerical implementations of the VO fractional problems are given. In \cite{Zhao-2015}, Zhao, Sun, and Karniadakis derived two second-order approximations for the VO Caputo fractional derivative and provided the error analysis. For the VO time-fractional sub-diffusion equations, Du, Alikhanov, and Sun \cite{Du-2020} proposed $L2$-$1_{\sigma}$ scheme, which makes use of the piecewise high-order polynomial interpolation of the solution. The resulting method was investigated to be unconditionally stable and second-order convergent via the energy method.

In consequence of the nonlocal property and historical dependence of the fractional operators, the aforementioned numerical methods always require all previous function values, which leads to an average $\mathcal{O}(n)$  storage and computational cost $\mathcal{O}(n^2)$, where $n$ is the total number of the time levels. To overcome this difficulty, many efforts have been made to speed up the evaluation of the CO fractional derivative \cite{Lu-2015, Lu-2018, Ke-2015, Fu-2017, Bertaccini-2019, Jiang-2017, Lubich-2002, Baffet-2016, Zeng-2017}. Nevertheless, the coefficient  matrices of the numerical schemes for the VO fractional problems lose the Toeplitz-like structure and the VO fractional derivative is no longer a convolution operator. Therefore, those fast methods for the CO fractional derivative cannot be directly applied to VO cases. Recently, Fang, Sun, and Wang \cite{Fang-2020} proposed a fast algorithm for the VO Caputo fractional derivative based on a shifted binary block partition  and uniform polynomial approximations of degree $r$. Compared with the general direct method, the proposed algorithm reduces the memory requirement from $\mathcal{O}(n)$ to $\mathcal{O}(r \log n)$ and the complexity
from $\mathcal{O}(n^2)$ to $\mathcal{O}(rn \log n)$. Zhang, Fang, and Sun \cite{Zhang-2020} developed an exponential-sum-approximation (ESA) technique to speed up  $L1$ formula for the VO Caputo fractional derivative. The proposed method reduces the storage requirement to $\mO(\log^2 n)$ and the computational cost to $\mathcal{O}(n\log^2 n)$, respectively. Both of those fast methods were proved to achieve the convergence order of $2-\overline{\alpha}$ with expected parameters.

Motivated by the need of fast high-order approaches for the VO Caputo fractional derivative (\ref{caputo-def}), we combine $L2$-$1_{\sigma}$ formula \cite{Du-2020} with the ESA technique \cite{Zhang-2020} to produce a fast evaluation formula, which is called $FL2$-$1_{\sigma}$ formula later. Compared with $L2$-$1_{\sigma}$ formula, the fast algorithm reduces the acting memory from $\mathcal{O}(n)$ to $\mO(\log^2n)$ and flops
from $\mathcal{O}(n^2)$ to $\mathcal{O}(n\log^2 n)$. Then $FL2$-$1_{\sigma}$ formula is applied to construct a fast temporal second-order and spatial fourth-order difference scheme ($FL2$-$1_{\sigma}$ scheme) for the multi-dimensional VO time-fractional sub-diffusion equations, which significantly reduces the memory requirement and computational complexity. We present the properties of the coefficients of $FL2$-$1_{\sigma}$ formula, which ensure the stability and the convergence of the proposed scheme. Numerical experiments are provided to show the sharp decrease in the CPU time and storage of the fast algorithm with the same accuracy as the direct method.

The paper is organized as follows. In Section \ref{fast-approximation}, we present $FL2$-$1_{\sigma}$ formula approaching the VO Caputo fractional derivative utilizing the ESA technique. Then we construct $FL2$-$1_{\sigma}$ scheme for multi-dimensional VO fractional sub-diffusion problems in Section \ref{finite-difference-scheme} and investigate the stability and convergence. In Section \ref{numerical-results}, reliability and efficiency are confirmed by some numerical examples. Concluding remarks are given in Section \ref{concluding-remarks}.

\section{Approximation for VO  fractional derivative}\label{fast-approximation}
We firstly introduce $L2$-$1_{\sigma}$ formula \cite{Du-2020} for the VO Caputo fractional derivative. Consider the VO Caputo fractional derivative defined by (\ref{caputo-def}) with $t\in(0,T]$ $(T\geq1)$. For a positive integer $n$, let $\Delta t=T/n$ and $t_k=k\Delta t$ for $k=0,1,\ldots,n$. We denote $t_{k+\sigma_k}=t_k+\sigma(t_k)\Delta t$ for $k=0,1,\ldots,n-1$. The parameter $\sigma(t_k)$ is determined by the equation \cite{Du-2020}
\begin{align*}
F(\sigma)=\sigma-\Big(1-\frac 1 2\alpha(t_k+\sigma\Delta t)\Big)=0,
\end{align*}
which has a unique root $\sigma(t_k)\in(\frac 1 2,1)$  being conveniently calculated by Newton's method. In the following, we denote $\sigma_k=\sigma(t_k)$, $\alpha_{k+\sigma_k}=\alpha(t_{k+\sigma_k})$.

From the definition \eqref{caputo-def}, the VO Caputo derivative at the point $t_{k+\sigma_k}$ is presented as
\begin{align}\label{caputo}
\prescript{C}{0}{\mathcal{D}}^{\alpha(t)}_tu(t)|_{t=t_{k+\sigma_k}}&=\frac 1 {\Gamma\big(1-\alpha_{k+\sigma_k}\big)}\int_{0}^{t_{k+\sigma_k}}\frac{u'(\tau)}{(t_{k+\sigma_k}-\tau)^{\alpha_{k+\sigma_k}}}\mathrm{d}\tau.
\end{align}
To discretize the VO Caputo fractional derivative, we denote the quadric interpolation on the time interval $[t_{k-1},t_k]$ with $1\leq k\leq n-1$ and the linear interpolation over the time interval $[t_k,t_{k+1}]$ with $0\leq k\leq n-1$ by
\begin{align*}
&L_2 u(\tau)=\sum_{p=-1}^1u(t_{k+p})\prod_{q=-1}^1\frac{\tau-t_{k+q}}{t_{k+p}-t_{k+q}},\ \ \tau\in [t_{k-1}, t_k],\\
&L_1u(\tau)=u(t_k)\frac {\tau-t_{k+1}}{t_k-t_{k+1}}+u(t_{k+1})\frac{\tau-t_k}{t_{k+1}-t_k},\ \ \tau\in [t_k, t_{k+1}].
\end{align*}

\subsection{$L2$-$1_{\sigma}$ formula}
Based on the above interpolation polynomials, at time $t_{k+\sigma_k}$ with $k=0,1,\ldots,n-1$, $L2$-$1_{\sigma}$ approximation formula to (\ref{caputo}) is obtained as\cite{Du-2020}
\begin{align}\label{L21}
\prescript{H}{0}{\mathcal{D}}^{\alpha_{k+\sigma_k}}_tu(t_{k+\sigma_k})
=&\frac 1 {\Gamma(1-\alpha_{k+\sigma_k})}\left(\int_{0}^{t_k}\frac{\big(L_2u(\tau)\big)'}{(t_{k+\sigma_k}-\tau)^{\alpha_{k+\sigma_k}}}\mathrm{d}\tau
  + \int_{t_k}^{t_{k+\sigma_k}}\frac{\big(L_1u(\tau)\big)'}{(t_{k+\sigma_k}-\tau)^{\alpha_{k+\sigma_k}}}\mathrm{d}\tau\right)\nonumber\\
%=&s^{(k)}\sum_{j=0}^k g_{k-j}^{(k)}[u(t_{j+1})-u(t_j)]\nonumber\\
=&s^{(k)}\sum_{l=0}^k g_l^{(k)}\big(u(t_{k-l+1})-u(t_{k-l})\big),
\end{align}
where $g_0^{(0)}=\sigma_0^{1-\alpha_{\sigma_0}}$, and for $k\geq 1$,
\begin{align*}
g_l^{(k)}&=\Delta t^{\alpha_{k+\sigma_k}-2}(1-\alpha_{k+\sigma_k})\cdot
\begin{cases}
\int_{t_{k-1}}^{t_k}\frac {\tau-t_{k-1/2}}{(t_{k+\sigma_k}-\tau)^{\alpha_{k+\sigma_k}}}\mathrm{d}\tau
   +\int_{t_k}^{t_{k+\sigma_k}}\frac {\Delta t}{(t_{k+\sigma_k}-\tau)^{\alpha_{k+\sigma_k}}}\mathrm{d}\tau,\ \ l=0,\\
\int_{t_{k-l-1}}^{t_{k-l}}\frac {\tau-t_{k-l-1/2}}{(t_{k+\sigma_k}-\tau)^{\alpha_{k+\sigma_k}}}\mathrm{d}\tau
   +\int_{t_{k-l}}^{t_{k-l+1}}\frac {t_{k-l+3/2}-\tau}{(t_{k+\sigma_k}-\tau)^{\alpha_{k+\sigma_k}}}\mathrm{d}\tau,\\
   ~~~~~~~~~~~~~~~~~~~~~~~~~~~~~~~~~~~~~~~~~~~~~~~~~~~~~~~~~1\leq l\leq k-1,\\
\int_{t_0}^{t_1}\frac {t_{3/2}-\tau}{(t_{k+\sigma_k}-\tau)^{\alpha_{k+\sigma_k}}}\mathrm{d}\tau,\ \ l=k,
  \end{cases}
%&=\frac {\Delta t^{-\alpha_{k+\sigma_k}}} {\Gamma(1-\alpha_{k+\sigma_k})}
%\left\{
%\begin{array}{ll}
%\int_0^1\frac {3/2-s}{(k+\sigma_k-s)^{\alpha_{k+\sigma_k}}}ds, \ \ \ j=0,\\
%\int_0^1\frac {s-1/2}{(k+\sigma_k-s-j+1)^{\alpha_{k+\sigma_k}}}ds+
%\int_0^1\frac {3/2-s}{(k+\sigma_k-s-j)^{\alpha_{k+\sigma_k}}}ds, \ \ \ 1\leq j\leq k-1,\\
%\int_0^1\frac {s-1/2}{(\sigma_k-s+1)^{\alpha_{k+\sigma_k}}}ds+
%\int_0^{\sigma_k}\frac 1{(\sigma_k-s)^{\alpha_{k+\sigma_k}}}ds, \ \ \ j=k.
%  \end{array}
%\right.
\end{align*}
%The coefficients can be simplified as $g_0^{(0)}=a_0^{(0)}$, and for $k\geq 1$,
%\begin{align*}
%g_j^{(k)}&=
%\left\{
%\begin{array}{ll}
%&a_k^{(k)}-b_k^{(k)}, \ \ \ \ \ \ \ \ \ \ \ \ \  \ \ \ \ \ \ \ \ j=0,\\
%&b_{k-j+1}^{(k)}+a_{k-j}^{(k)}-b_{k-j}^{(k)}, \ \ \  \ \ \ 1\leq j\leq k-1,\\
%&a_0^{(k)}+b_1^{(k)}, \ \ \ \ \ \ \ \ \ \ \ \ \ \ \ \ \ \ \ \ \ j=k,
%  \end{array}
%\right.
%\end{align*}
and
\begin{align*}
s^{(k)}=&\frac {\Delta t^{-\alpha_{k+\sigma_k}}}{\Gamma(2-\alpha_{k+\sigma_k})}.
%,\ \ \ \ \ \ \ a_0^{(k)}=\sigma_k^{1-\alpha_{k+\sigma_k}},\\ a_p^{(k)}=&(p+\sigma_k)^{1-\alpha_{k+\sigma_k}}-(p-1+\sigma_k)^{1-\alpha_{k+\sigma_k}},\\
%b_p^{(k)}=&\frac 1 {2-\alpha_{k+\sigma_k}}[(p+\sigma_k)^{2-\alpha_{k+\sigma_k}}-(p-1+\sigma_k)^{2-\alpha_{k+\sigma_k}}]\\
%          &-\frac 1 2[(p+\sigma_k)^{1-\alpha_{k+\sigma_k}}+(p-1+\sigma_k)^{1-\alpha_{k+\sigma_k}}],\ \ \ \ \ \ \ \ \ p\geq1.
\end{align*}

The following lemma shows the local truncation error of $L2$-$1_{\sigma}$ approximation formula for the VO Caputo fractional derivative.
\begin{lemma}\label{L-error-L21}\cite{Du-2020}
Suppose $u\in \mathcal{C} ^3\big([0,t_{k+1}]\big)$. Let
\begin{align*}
r^k=\Big|\prescript{C}{0}{\mathcal{D}}^{\alpha(t)} _tu(t)|_{t=t_{k+\sigma_k}}-\prescript{H}{0}{\mathcal{D}}^{\alpha_{k+\sigma_k}}_tu(t_{k+\sigma_k})\Big|.
\end{align*}
Then, we have
\begin{align*}
\big|r^k\big|
\leq\frac{\max\limits_{0\leq t\leq t_{k+1}}\big|u^{(3)}(t)\big|\sigma_k^{-\alpha_{k+\sigma_k}}}{\Gamma(1-\alpha_{k+\sigma_k})}
    \left(\frac 1{12}+\frac{\sigma_k}{6(1-\alpha_{k+\sigma_k})} \right)\Delta t^{3-\alpha_{k+\sigma_k}}.
\end{align*}
\end{lemma}

Utilizing $L2$-$1_{\sigma}$ approximation formula to calculate the value at the current time level, it needs to compute the summation of a series including the values of all previous time levels. Therefore, $L2$-$1_{\sigma}$ approximation formula requires $\mathcal{O}(n)$ storage and $\mathcal{O}(n^2)$ computational complexity. It inspires us to construct a fast algorithm to approach $L2$-$1_{\sigma}$ approximation formula (\ref{L21}).

\subsection{$FL2$-$1_{\sigma}$ formula}
Now, we develop a fast high-order numerical formula ($FL2$-$1_{\sigma}$ formula) for the VO Caputo fractional derivative. The kernel function in the VO Caputo fractional derivative is approximated by the ESA technique, which is stated in \cite{Beylkin-2017, Zhang-2020}.

\begin{lemma}\label{L-soe}
For any constant $\alpha_{k+\sigma_k}\in[\underline{\alpha},\overline{\alpha}]\subset(0,1)$, $0<\frac{\Delta t}T\leq \frac{t_{k+\sigma_k}-\tau}T\leq1$ for $\tau\in[0,t_{k-1}]$, $1\leq k \leq n-1$ and the expected accuracy $0<\epsilon\leq\ 1/e$, there exist a constant $h$, integers $\overline{N}$ and $\underline{N}$, which satisfy
\begin{align}\label{parameter-choose}
h&=\frac {2\pi}{\log3+\overline{\alpha}\log(\cos 1)^{-1}+\log\epsilon^{-1}},\nonumber\\
\underline{N}&=\left\lceil\frac 1{h}\frac 1{\underline{\alpha}}\big(\log\epsilon+\log\Gamma(1+\overline{\alpha})\big)\right\rceil,\\
\overline{N}&=\left\lfloor\frac 1{h}\left(\log \frac T{\Delta t}+\log\log\epsilon^{-1}+\log\underline{\alpha}+2^{-1}\right)\right\rfloor,\nonumber
\end{align}
such that
\begin{align}\label{soe}
\Bigg|\left(\frac {t_{k+\sigma_k}-\tau}{T}\right)^{-\alpha_{k+\sigma_k}}-\sum_{i=\underline{N}+1}^{\overline{N}}\theta_i^{(k)} e^{-\lambda_i(t_{k+\sigma_k}-\tau)/T}\Bigg|
\leq \left(\frac {t_{k+\sigma_k}-\tau}{T}\right)^{-\alpha_{k+\sigma_k}}\epsilon,
\end{align}
where the quadrature exponents and weights are given by
\begin{align*}
\lambda_i=e^{ih},\ \ \theta_i^{(k)}=\frac{he^{\alpha_{k+\sigma_k}ih}}{\Gamma(\alpha_{k+\sigma_k})}. %\hat{\theta}_i^{(k)}=\frac{T^{-\alpha_{k+\sigma_k}}}{\Gamma(1-\alpha_{k+\sigma_k})}\theta_i^{(k)},
\end{align*}
\end{lemma}

%For $1\leq k \leq n-1$, we deduce $FL2$-$1_{\sigma}$ formula via approximating $\left(\frac {t_{k+\sigma_k}-\tau}{T}\right)^{-\alpha_{k+\sigma_k}}$ in the Caputo fractional derivative by the ESA \cite{Beylkin-2017, Zhang-2020} with an expected accuracy $\epsilon$; i.e.,
%\begin{align}\label{soe}
%\big|\left(\frac {t_{k+\sigma_k}-\tau}{T}\right)^{-\alpha_{k+\sigma_k}}-\sum_{i=\underline{N}+1}^{\overline{N}}\theta_i^{(k)} e^{-\lambda_i(t_{k+\sigma_k}-\tau)/T}\big|
%\leq \left(\frac {t_{k+\sigma_k}-\tau}{T}\right)^{-\alpha_{k+\sigma_k}}\epsilon,
%\end{align}
%where the quadrature exponents and weights are given by
%\begin{align*}
%\lambda_i=e^{ih},\ \ \theta_i^{(k)}=\frac{he^{\alpha_{k+\sigma_k}ih}}{\Gamma(\alpha_{k+\sigma_k})},\ \ %\hat{\theta}_i^{(k)}=\frac{T^{-\alpha_{k+\sigma_k}}}{\Gamma(1-\alpha_{k+\sigma_k})}\theta_i^{(k)},
%\end{align*}
%in which $h$, $\underline{N}$ and $\overline{N}$ are properly chosen as follows
%\begin{align}\label{parameter-choose}
%h&=\frac {2\pi}{\log3+\overline{\alpha}\log(\cos 1)^{-1}+\log\epsilon^{-1}},\nonumber\\
%\underline{N}&=\left\lceil\frac 1{h}\frac 1{\underline{\alpha}}\left[\log\epsilon+\log\Gamma(1+\overline{\alpha})\right]\right\rceil,\\
%\overline{N}&=\left\lfloor\frac 1{h}\left(\log \frac T{\Delta t}+\log\log\epsilon^{-1}+\log\underline{\alpha}+2^{-1}\right)\right\rfloor.\nonumber
%\end{align}
Now $L2$-$1_{\sigma}$ formula \eqref{L21} can be approximated by {\small
\begin{align}\label{HL-parts}
&\prescript{H}{0}{\mathcal{D}}^{\alpha_{k+\sigma_k}}_tu(t_{k+\sigma_k})\nonumber\\
=&\frac 1 {\Gamma(1-\alpha_{k+\sigma_k})}\Bigg(T^{-\alpha_{k+\sigma_k}}\int_{0}^{t_k}\big(L_2u(\tau)\big)'
  \left(\frac{t_{k+\sigma_k}-\tau}{T}\right)^{-\alpha_{k+\sigma_k}} \mathrm{d}\tau
  + \int_{t_k}^{t_{k+\sigma_k}}\frac{\big(L_1 u(\tau)\big)'}{(t_{k+\sigma_k}-\tau)^{\alpha_{k+\sigma_k}}}\mathrm{d}\tau\Bigg)\nonumber\\
\approx&\frac 1 {\Gamma(1-\alpha_{k+\sigma_k})}\Bigg(T^{-\alpha_{k+\sigma_k}}\int_{0}^{t_k}\big(L_2u(\tau)\big)'
  \sum_{i=\underline{N}+1}^{\overline{N}}\theta_{i}^{(k)}e^{-\lambda_i(t_{k+\sigma_k}-\tau)/T}\mathrm{d}\tau
  +\int_{t_k}^{t_{k+\sigma_k}}\frac{\big(L_1u(\tau)\big)'}{(t_{k+\sigma_k}-\tau)^{\alpha_{k+\sigma_k}}}\mathrm{d}\tau\Bigg)\nonumber\\
=&\frac{T^{-\alpha_{k+\sigma_k}}}{\Gamma(1-\alpha_{k+\sigma_k})}
  \sum_{i=\underline{N}+1}^{\overline{N}}\theta_{i}^{(k)}\int_{0}^{t_k}\big(L_2u(\tau)\big)' e^{-\lambda_i(t_{k+\sigma_k}-\tau)/T}\mathrm{d}\tau
  +\frac{u(t_{k+1})-u(t_k)}{\Delta t\Gamma(1-\alpha_{k+\sigma_k})}
  \int_{t_k}^{t_{k+\sigma_k}}\frac 1{(t_{k+\sigma_k}-\tau)^{\alpha_{k+\sigma_k}}}\mathrm{d}\tau\nonumber\\
=&\frac{T^{-\alpha_{k+\sigma_k}}}{\Gamma(1-\alpha_{k+\sigma_k})}
  \sum_{i=\underline{N}+1}^{\overline{N}}\theta_{i}^{(k)}H_i^{(k)} + s^{(k)}\sigma_k^{1-\alpha_{k+\sigma_k}}\big(u(t_{k+1})-u(t_k)\big),
%:=&\prescript{FH}{0}{\mathcal{D}}^{\alpha_{k+\sigma_k}}_tu(t_{k+\sigma_k})
\end{align} }
where $H_i^{(k)}$ is the history part of the integral,
\begin{align*}
H_i^{(k)}=\int_{0}^{t_k}\big(L_2u(\tau)\big)' e^{-\lambda_i(t_{k+\sigma_k}-\tau)/T}\mathrm{d}\tau.
\end{align*}
We point out that $H_i^{(k)}$ can be calculated by a recursive relation and quadratic interpolation; i.e.,
\begin{align}\label{his}
H_i^{(k)}=&e^{-\lambda_i(1+\sigma_k-\sigma_{k-1})\Delta t/T} H_i^{(k-1)}+\int_{t_{k-1}}^{t_k}\big(L_2u(\tau)\big)'e^{-\lambda_i(t_{k+\sigma_k}-\tau)/T}\mathrm{d}\tau\nonumber\\
=&e^{-\lambda_i(1+\sigma_k-\sigma_{k-1})\Delta t/T} H_i^{(k-1)}+A_i^{(k)}\big(u(t_k)-u(t_{k-1})\big)+B_i^{(k)}\big(u(t_{k+1})-u(t_k)\big),
\end{align}
with $H_i^{(0)}=0$ $(i=\underline{N}+1,\ldots,\overline{N})$ and
\begin{align*}
A_i^{(k)}=\int_0^1\big(3/2-\tau\big)e^{-\lambda_i(\sigma_k+1-\tau)\Delta t/T}\mathrm{d}\tau,\ \ \ \
B_i^{(k)}=\int_0^1\big(\tau-1/2\big)e^{-\lambda_i(\sigma_k+1-\tau)\Delta t/T}\mathrm{d}\tau.
\end{align*}

Overall, for $0\leq k \leq n-1$, we define $FL2$-$1_{\sigma}$ formula for $\prescript{C}{0}{\mathcal{D}}^{\alpha(t)}_tu(t)|_{t=t_{k+\sigma_k}}$ by $\prescript{FH}{0}{\mathcal{D}}^{\alpha_{k+\sigma_k}}_tu(t_{k+\sigma_k})$ given as
\begin{align}
\prescript{FH}{0}{\mathcal{D}}^{\alpha_{k+\sigma_k}}_tu(t_{k+\sigma_k})
=\frac{T^{-\alpha_{k+\sigma_k}}}{\Gamma(1-\alpha_{k+\sigma_k})}\sum_{i=\underline{N}+1}^{\overline{N}}\theta_{i}^{(k)}H_i^{(k)}
 + s^{(k)}\sigma_k^{1-\alpha_{k+\sigma_k}}\big(u(t_{k+1})-u(t_k)\big),\label{FL21}
\end{align}
where $H_i^{(k)}$ is calculated by \eqref{his}.

In addition, for $k=0$, we let
\begin{align}\label{FL21-0}
\prescript{FH}{0}{\mathcal{D}}^{\alpha_{\sigma_0}}_tu(t_{\sigma_0})=\prescript{H}{0}{\mathcal{D}}^{\alpha_{\sigma_0}}_tu(t_{\sigma_0}).
\end{align}

We give the following lemma to state the local truncated error of $FL2$-$1_{\sigma}$ formula \eqref{FL21} for the VO Caputo fractional derivative $\prescript{C}{0}{\mathcal{D}}^{\alpha(t)}_tu(t)|_{t=t_{k+\sigma_k}}$.
\begin{lemma}\label{L-error-FL21}
Suppose $\alpha_{k+\sigma_k}\in(0,1)$ and $u(t)\in \mathcal{C}^3\big([0,T]\big)$. Let $L2$-$1_{\sigma}$ formula be as in (\ref{L21}), $FL2$-$1_{\sigma}$ formula be defined by \eqref{FL21}--\eqref{FL21-0} and $\epsilon$ be the expected accuracy. Then, we have
\begin{align}\label{error-FL21}
\prescript{C}{0}{\mathcal{D}}^{\alpha(t)}_tu(t)|_{t=t_{k+\sigma_k}}
=\prescript{FH}{0}{\mathcal{D}}^{\alpha_{k+\sigma_k}}_tu(t_{k+\sigma_k})+\mathcal{O}(\Delta t^{3-\alpha_{k+\sigma_k}}+\epsilon),\ \ \ \ 0\leq k \leq n-1.
\end{align}
\end{lemma}
\begin{proof}
Clearly, (\ref{FL21-0}) and Lemma \ref{L-error-L21} imply the lemma is valid for $k=0$. For $k\geq1$, according to \eqref{HL-parts} and \eqref{FL21}, $\prescript{H}{0}{\mathcal{D}}^{\alpha_{k+\sigma_k}}_tu(t_{k+\sigma_k})$ and $\prescript{FH}{0}{\mathcal{D}}^{\alpha_{k+\sigma_k}}_tu(t_{k+\sigma_k})$ just differ in the approximation for $\left(\frac{t_{k+\sigma_k}-\tau}T\right)^{-\alpha_{k+\sigma_k}}$, $\tau\in(0,t_{k+\sigma_k})$; i.e.,
\begin{align*}
&\Big|\prescript{C}{0}{\mathcal{D}}^{\alpha(t)}_tu(t)|_{t=t_{k+\sigma_k}}-\prescript{FH}{0}{\mathcal{D}}^{\alpha_{k+\sigma_k}}_tu(t_{k+\sigma_k})\Big|\\
\leq&\Big|\prescript{C}{0}{\mathcal{D}}^{\alpha(t)}_tu(t)|_{t=t_{k+\sigma_k}}-\prescript{H}{0}{\mathcal{D}}^{\alpha_{k+\sigma_k}}_tu(t_{k+\sigma_k})\Big|
     +\Big|\prescript{H}{0}{\mathcal{D}}^{\alpha_{k+\sigma_k}}_tu(t_{k+\sigma_k})-\prescript{FH}{0}{\mathcal{D}}^{\alpha_{k+\sigma_k}}_tu(t_{k+\sigma_k})\Big|\\
=&\mathcal{O}(\Delta t^{3-\alpha_{k+\sigma_k}})\\
  &+\frac {T^{-\alpha_{k+\sigma_k}}} {\Gamma(1-\alpha_{k+\sigma_k})}\sum_{l=1}^k\int_{t_{l-1}}^{t_l}
   \Bigg|\left(\frac{t_{k+\sigma_k}-\tau}T\right)^{-\alpha_{k+\sigma_k}}-\sum_{i=\underline{N}+1}^{\overline{N}}\theta_i^{(k)}e^{-\lambda_i(t_{k+\sigma_k}-\tau)/T}\Bigg|
   \big(L_2u(\tau)\big)'\mathrm{d}\tau\\
=&\mathcal{O}(\Delta t^{3-\alpha_{k+\sigma_k}})+\frac {t_k^{1-\alpha_{k+\sigma_k}}\max\limits_{0\leq t\leq t_k}|u'(t)|} {\Gamma(2-\alpha_{k+\sigma_k})}\mO(\epsilon).
\end{align*}
Since $t_{k+\sigma_k}\leq T$ and  $u(t)\in \mathcal{C}^3\big([0,T]\big)$, the proof is completed.
\end{proof}

From Lemma \ref{L-error-FL21}, we conclude that the approximation for the VO Caputo fractional derivative $\prescript{C}{0}{\mathcal{D}}^{\alpha(t)}_tu(t)$ at the point $t_{k+\sigma_k}$ by $\prescript{FH}{0}{\mathcal{D}}^{\alpha_{k+\sigma_k}}_tu(t_{k+\sigma_k})$ is at least second-order convergent if $u\in \mathcal{C}^3\big([0,T]\big)$.

\subsection{Properties of discrete kernels}
In order to prepare the subsequent analysis, we firstly show the properties of discrete kernels in $FL2$-$1_{\sigma}$ formula. The recursive relation of $H_i^{(k)}$ defined in \eqref{his} can be equivalently rewritten as the following form {\small
\begin{align*}
H_i^{(k)}
=&\sum_{l=1}^k\int_{t_{l-1}}^{t_l}\big(L_2u(\tau)\big)' e^{-\lambda_i(t_{k+\sigma_k}-\tau)/T}\mathrm{d}\tau\\
=&\Delta t^{-2}\bigg\{\int_{t_0}^{t_1}(t_{\frac 3 2}-\tau) e^{-\lambda_i(t_{k+\sigma_k}-\tau)/T}\mathrm{d}\tau\big(u(t_1)-u(t_0)\big)\\
&+\sum_{l=1}^{k-1}\left[\int_{t_{l-1}}^{t_l}(\tau-t_{l-\frac 1 2}) e^{-\lambda_i(t_{k+\sigma_k}-\tau)/T}\mathrm{d}\tau
 +\int_{t_l}^{t_{l+1}}(t_{l+\frac 3 2}-\tau) e^{-\lambda_i(t_{k+\sigma_k}-\tau)/T}\mathrm{d}\tau\right]\big(u(t_{l+1})-u(t_l)\big)\\
&+\int_{t_{k-1}}^{t_k}(\tau-t_{k-\frac 1 2}) e^{-\lambda_i(t_{k+\sigma_k}-\tau)/T}\mathrm{d}\tau\big(u(t_{k+1})-u(t_k)\big)\bigg\},
%=&e^{-\lambda_i(k-1)\Delta t/T}A_i^{(k)}\left[u(t_1)-u(t_0)\right]\nonumber\\
%&+\sum_{j=1}^{k-1}[e^{-\lambda_i(k-j-1)\Delta t/T}A_i^{(k)}+e^{-\lambda_i(k-j)\Delta t/T}B_i^{(k)}]\left[u(t_{j+1})-u(t_j)\right]\nonumber\\
%&+B_i^{(k)}\left[u(t_{k+1})-u(t_{k})\right],
\end{align*} }
which leads to $\prescript{FH}{0}{\mathcal{D}}^{\alpha_{k+\sigma_k}}_tu(t_{k+\sigma_k})$ $(0\leq k\leq n-1)$ in the following {\small
\begin{align}
&\prescript{FH}{0}{\mathcal{D}}^{\alpha_{k+\sigma_k}}_tu(t_{k+\sigma_k})\nonumber\\
=&\frac{T^{-\alpha_{k+\sigma_k}}\Delta t^{-2}}{\Gamma(1-\alpha_{k+\sigma_k})}\sum_{i=\underline{N}+1}^{\overline{N}}\theta_{i}^{(k)}
  \bigg\{\int_{t_0}^{t_1}(t_{\frac 3 2}-\tau) e^{-\lambda_i(t_{k+\sigma_k}-\tau)/T}\mathrm{d}\tau\big(u(t_1)-u(t_0)\big)\nonumber\\
&+\sum_{l=1}^{k-1}\left[\int_{t_{l-1}}^{t_l}(\tau-t_{l-\frac 1 2}) e^{-\lambda_i(t_{k+\sigma_k}-\tau)/T}\mathrm{d}\tau
 +\int_{t_l}^{t_{l+1}}(t_{l+\frac 3 2}-\tau) e^{-\lambda_i(t_{k+\sigma_k}-\tau)/T}\mathrm{d}\tau\right]\big(u(t_{l+1})-u(t_l)\big)\nonumber\\
&+\int_{t_{k-1}}^{t_k}(\tau-t_{k-\frac 1 2}) e^{-\lambda_i(t_{k+\sigma_k}-\tau)/T}\mathrm{d}
  \tau\big(u(t_{k+1})-u(t_k)\big)\bigg\}+s^{(k)}\sigma_k^{1-\alpha_{k+\sigma_k}}\big(u(t_{k+1})-u(t_{k})\big)\nonumber\\
%=&\sum_{i=\underline{N}+1}^{\overline{N}}\hat{\theta}_{i}^{(k)}
%\{e^{-\lambda_i(k-1)\Delta t/T}A_i^{(k)}\left[u(t_1)-u(t_0)\right]\nonumber\\
%&+\sum_{j=1}^{k-1}\left(e^{-\lambda_i(k-j-1)\Delta t/T}A_i^{(k)}+e^{-\lambda_i(k-j)\Delta t/T}B_i^{(k)}\right)\left[u(t_{j+1})-u(t_j)\right]\nonumber\\
%&+B_i^{(k)}\left[u(t_{k+1})-u(t_{k})\right]\}+s^{(k)}a_0^{(k)}\left[u(t_{k+1})-u(t_{k})\right]\nonumber\\
=&s^{(k)}\sum_{l=0}^k \rho_{l}^{(k)}\big(u(t_{k-l+1})-u(t_{k-l})\big),\label{FL21-2}
\end{align} }
where
\begin{align}\label{coe-FL21-1}
\rho_0^{(0)}=\sigma_0^{1-\alpha_{\sigma_0}},
\end{align}
and for $1 \leq k\leq n-1$, {\small
\begin{align}\label{coe-FL21}
\rho_l^{(k)}
&=\frac{\Delta t^{\alpha_{k+\sigma_k}-2}(1-\alpha_{k+\sigma_k})}{T^{\alpha_{k+\sigma_k}}}\cdot
\begin{cases}
\int_{t_{k-1}}^{t_k}(\tau-t_{k-\frac 1 2})\sum\limits_{i=\underline{N}+1}^{\overline{N}}\theta_{i}^{(k)} e^{-\lambda_i(t_{k+\sigma_k}-\tau)/T}\mathrm{d}\tau\\
\ \ \ \ +T^{\alpha_{k+\sigma_k}}\int_{t_k}^{t_{k+\sigma_k}}\frac{\Delta t}{(t_{k+\sigma_k}-\tau)^{\alpha_{k+\sigma_k}}}\mathrm{d}\tau,\ \ l=0,\\
\int_{t_{k-l-1}}^{t_{k-l}}(\tau-t_{k-l-\frac 1 2})\sum\limits_{i=\underline{N}+1}^{\overline{N}}\theta_{i}^{(k)} e^{-\lambda_i(t_{k+\sigma_k}-\tau)/T}\mathrm{d}\tau\\
\ \ \ \ +\int_{t_{k-l}}^{t_{k-l+1}}(t_{k-l+\frac 3 2}-\tau)\sum\limits_{i=\underline{N}+1}^{\overline{N}}\theta_{i}^{(k)}e^{-\lambda_i(t_{k+\sigma_k}-\tau)/T}\mathrm{d}\tau,\\
~~~~~~~~~~~~~~~~~~~~~~~~~~~~~~~~~~~~~~~~~~~~~~~~~~~~1\leq l\leq k-1,\\
\int_{t_0}^{t_1}(t_{\frac 3 2}-\tau)\sum\limits_{i=\underline{N}+1}^{\overline{N}}\theta_{i}^{(k)}e^{-\lambda_i(t_{k+\sigma_k}-\tau)/T}\mathrm{d}\tau,\ \ l=k.
  \end{cases}
\end{align} }

The coefficients $\{\rho_l^{(k)}|0\leq l\leq k; 0\leq k\leq n-1\}$ defined in \eqref{coe-FL21-1} and \eqref{coe-FL21} have the following properties, which play the vital roles in   studying   the stability and convergence of the finite difference schemes. Below we present the relationship between $\rho_l^{(k)}$ and $\prescript{}{}{g}_l^{(k)}$.

\begin{lemma}\label{coefficient1}
For $\alpha_{k+\sigma_k}\in(0,1)$, $g_l^{(k)}$ $(0\leq l\leq k; 0\leq k\leq n-1)$ defined in \eqref{L21}, $\rho_l^{(k)}$ $(0\leq l\leq k; 0\leq k\leq n-1)$ defined in \eqref{coe-FL21-1} and \eqref{coe-FL21} with $\sigma_k=1-\frac {\alpha_{k+\sigma_k}}2$. Then, we have
\begin{align*}
\rho_0^{(0)}=\prescript{ }{}{g}_0^{(0)},
\end{align*}
and for $1\leq k\leq n-1$,
\begin{align}\label{coe}
\Big|\rho_l^{(k)}-\prescript{ }{}{g}_l^{(k)}\Big|
\leq\frac{1-\alpha_{k+\sigma_k}}{\Delta t^{\alpha_{k+\sigma_k}}}\cdot
\begin{cases}
\frac \epsilon 4,\ \ l=0,\\
\frac {5\epsilon} 4,\ \ 1\leq l\leq k-1,\\
\ \epsilon,\ \ l=k.
\end{cases}
\end{align}
\end{lemma}

\begin{proof}
For $k=0$, $\rho_0^{(0)}=g_0^{(0)}$ is obvious. For $k\geq 1$, the ESA approximation \eqref{soe} with \eqref{L21} and \eqref{FL21-2} gives
\begin{align*}
\Big|\rho_0^{(k)}-g_0^{(k)}\Big|
\leq&\frac{\Delta t^{\alpha_{k+\sigma_k}-2}(1-\alpha_{k+\sigma_k})}{T^{\alpha_{k+\sigma_k}}}\cdot \Bigg\{\int_{t_{k-1}}^{t_k}
     \big|\tau-t_{k-\frac 1 2}\big|\\
     &\cdot\bigg[\Big(\frac{t_{k+\sigma_k}-\tau}{T}\Big)^{-\alpha_{k+\sigma_k}}
     -\sum\limits_{i=\underline{N}+1}^{\overline{N}}\theta_{i}^{(k)}e^{-\lambda_i(t_{k+\sigma_k}-\tau)\Delta t/T}\bigg]\mathrm{d}\tau\Bigg\}\nonumber\\
\leq&\Delta t^{\alpha_{k+\sigma_k}-2}(1-\alpha_{k+\sigma_k})\epsilon\cdot\int_{t_{k-1}}^{t_k}\big|\tau-t_{k-\frac 1 2}\big|
     (t_{k+\sigma_k}-\tau)^{-\alpha_{k+\sigma_k}}\mathrm{d}\tau\nonumber\\
=&(1-\alpha_{k+\sigma_k})\epsilon\cdot\int_0^1\big|\tau-1/2\big|(\sigma_k-\tau+1)^{-\alpha_{k+\sigma_k}}\mathrm{d}\tau\nonumber\\
\leq&(1-\alpha_{k+\sigma_k})\Delta t^{-\alpha_{k+\sigma_k}}\epsilon\cdot\int_0^1\big|\tau-1/2\big|\mathrm{d}\tau\nonumber\\
=&\frac \epsilon 4(1-\alpha_{k+\sigma_k})\Delta t^{-\alpha_{k+\sigma_k}}.
\end{align*}

For $1\leq l\leq k-1$, we obtain
\begin{align*}
&\Big|\rho_l^{(k)}-g_l^{(k)}\Big|\nonumber\\
\leq&\frac{\Delta t^{\alpha_{k+\sigma_k}-2}(1-\alpha_{k+\sigma_k})}{T^{\alpha_{k+\sigma_k}}}\cdot\Bigg\{\int_{t_{k-l-1}}^{t_{k-l}}
     \big|\tau-t_{k-l-\frac 1 2}\big|\cdot\\
    &\bigg[\Big(\frac{t_{k+\sigma_k}-\tau}{T}\Big)^{-\alpha_{k+\sigma_k}}
     -\sum\limits_{i=\underline{N}+1}^{\overline{N}}\theta_{i}^{(k)}e^{-\lambda_i(t_{k+\sigma_k}-\tau)\Delta t/T}\bigg]\mathrm{d}\tau\nonumber\\
&+\int_{t_{k-l}}^{t_{k-l+1}}\big|t_{k-l+\frac 3 2}-\tau\big|\bigg[\Big(\frac{t_{k+\sigma_k}-\tau}{T}\Big)^{-\alpha_{k+\sigma_k}}
 -\sum\limits_{i=\underline{N}+1}^{\overline{N}}\theta_{i}^{(k)}e^{-\lambda_i(t_{k+\sigma_k}-\tau)\Delta t/T}\bigg]\mathrm{d}\tau\Bigg\}\nonumber\\
\leq&\Delta t^{\alpha_{k+\sigma_k}-2}(1-\alpha_{k+\sigma_k})\epsilon\cdot\Bigg\{\int_{t_{k-l-1}}^{t_{k-l}}\big|\tau-t_{k-l-\frac 1 2}\big|
     (t_{k+\sigma_k}-\tau)^{-\alpha_{k+\sigma_k}}\mathrm{d}\tau\nonumber\\
&+\int_{t_{k-l}}^{t_{k-l+1}}\big|t_{k-l+\frac 3 2}-\tau\big|(t_{k+\sigma_k}-\tau)^{-\alpha_{k+\sigma_k}}\mathrm{d}\tau\Bigg\}\nonumber\\
=&(1-\alpha_{k+\sigma_k})\epsilon\cdot\Big\{\int_0^1\big|\tau-1/2\big|(\sigma_k-\tau+l+1)^{-\alpha_{k+\sigma_k}}\mathrm{d}\tau
  +\int_0^1\big|3/2-\tau\big|(\sigma_k-\tau+l)^{-\alpha_{k+\sigma_k}}\mathrm{d}\tau\Big\}\nonumber\\
%=&\frac {\epsilon\Delta t^{-\alpha_{k+\sigma_k}} T^{-\alpha_{k+\sigma_k}}} {\Gamma(1-\alpha_{k+\sigma_k})}\{\int_0^1|\tau-\frac 1 2|(\frac{k+\sigma_k-\tau-j+1}{T})^{-\alpha_{k+\sigma_k}}\mathrm{d}\tau\nonumber\\
%&+\int_0^1|\frac 3 2-\tau|(\frac{k+\sigma_k-\tau-j}{T})^{-\alpha_{k+\sigma_k}}\mathrm{d}\tau\}\nonumber\\
\leq&(1-\alpha_{k+\sigma_k})\Delta t^{-\alpha_{k+\sigma_k}}\epsilon\cdot\Big\{\int_0^1\big|3/2-\tau\big|\mathrm{d}\tau+\int_0^1\big|\tau
     -1/2\big|\mathrm{d}\tau\nonumber\Big\}\\
=&\frac {5\epsilon} 4(1-\alpha_{k+\sigma_k})\Delta t^{-\alpha_{k+\sigma_k}}.
\end{align*}
Similarly, we have
\begin{align*}
&\Big|\rho_k^{(k)}-g_k^{(k)}\Big|\nonumber\\
\leq&\frac{\Delta t^{\alpha_{k+\sigma_k}-2}(1-\alpha_{k+\sigma_k})}{T^{\alpha_{k+\sigma_k}}}\cdot\Bigg\{\int_{t_0}^{t_1}\big|t_{\frac 3 2}-\tau\big|
     \bigg[\Big(\frac{t_{k+\sigma_k}-\tau}{T}\Big)^{-\alpha_{k+\sigma_k}}
     -\sum\limits_{i=\underline{N}+1}^{\overline{N}}\theta_{i}^{(k)}e^{-\lambda_i(t_{k+\sigma_k}-\tau)\Delta t/T}\bigg]\mathrm{d}\tau\Bigg\}\nonumber\\
\leq&\Delta t^{\alpha_{k+\sigma_k}-2}(1-\alpha_{k+\sigma_k})\epsilon\cdot\int_{t_0}^{t_1}\big|t_{\frac 3 2}
     -\tau\big|(t_{k+\sigma_k}-\tau)^{-\alpha_{k+\sigma_k}}\mathrm{d}\tau\nonumber\\
=&(1-\alpha_{k+\sigma_k})\epsilon \cdot\int_0^1\big|3/2-\tau\big|(k+\sigma_k-\tau)^{-\alpha_{k+\sigma_k}}\mathrm{d}\tau\nonumber\\
%=&\frac {\epsilon\Delta t^{-\alpha_{k+\sigma_k}} T^{-\alpha_{k+\sigma_k}}} {\Gamma(1-\alpha_{k+\sigma_k})}\int_0^1|\frac 3 2-\tau|(\frac{k+\sigma_k-\tau}{T})^{-\alpha_{k+\sigma_k}}\mathrm{d}\tau\nonumber\\
\leq&(1-\alpha_{k+\sigma_k})\Delta t^{-\alpha_{k+\sigma_k}}\epsilon\cdot\int_0^1\big|3/2-\tau\big|\mathrm{d}\tau\nonumber\\
=&\epsilon(1-\alpha_{k+\sigma_k})\Delta t^{-\alpha_{k+\sigma_k}}.
\end{align*}
The proof is complete.
\end{proof}

\begin{lemma}\label{L-coefficient}
For $\alpha_{k+\sigma_k}\in(0,1)$, $\rho_l^{(k)}$ $(0\leq l\leq k; 0\leq k\leq n-1)$ are defined in \eqref{coe-FL21-1} and \eqref{coe-FL21} with $\sigma_k=1-\frac {\alpha_{k+\sigma_k}}2$ and a sufficiently small $\epsilon$ satisfying
\begin{align*}
\epsilon
\leq\frac{2(1-\overline{\alpha})(2-\frac 1 2 \overline{\alpha})^{1-\overline{\alpha}}\Delta t^{\overline{\alpha}}}{(6-\frac7 2\underline{\alpha})(1-\frac 1 2 \underline{\alpha})}.
\end{align*}
Then, we have
\begin{align}
&0<\frac{(1-\epsilon)(1-\alpha_{k+\sigma_k})} {2(k+\sigma_k)^{\alpha_{k+\sigma_k}}}<\rho_k^{(k)}<\rho_{k-1}^{(k)}<\ldots<\rho_0^{(k)},\label{coe1}\\
&(2\sigma_k-1)\rho_0^{(k)}-\sigma_k\rho_1^{(k)}\geq0.\label{coe2}
\end{align}
\end{lemma}
\begin{proof}
For $k=0$, $\rho_0^{(0)}=g_0^{(0)}>0$ is obvious. For $k\geq 1$, from \eqref{coe-FL21}, the coefficients can be rewritten as the following form
\begin{align*}
\rho_l^{(k)}
&=\frac{\Delta t^{\alpha_{k+\sigma_k}}(1-\alpha_{k+\sigma_k})}{T^{\alpha_{k+\sigma_k}}}
\left\{
\begin{array}{ll}
\sum\limits_{i=\underline{N}+1}^{\overline{N}}\theta_{i}^{(k)}B_i^{(k)}
+\Delta t^{-1}T^{\alpha_{k+\sigma_k}}\int_{t_k}^{t_{k+\sigma_k}}\frac{1}{(t_{k+\sigma_k}-\tau)^{\alpha_{k+\sigma_k}}}\mathrm{d}\tau, \ \ l=0,\\
\sum\limits_{i=\underline{N}+1}^{\overline{N}}\theta_{i}^{(k)}\Big(e^{-\lambda_i(l-1)\Delta t/T}A_i^{(k)}
+e^{-\lambda_il\Delta t/T}B_i^{(k)}\Big),\ \ 1\leq l\leq k-1,\\
\sum\limits_{i=\underline{N}+1}^{\overline{N}}\theta_{i}^{(k)}e^{-\lambda_i(k-1)\Delta t/T}A_i^{(k)},\ \ l=k,
  \end{array}
\right.
\end{align*}
where $A_i^{(k)}$ and $B_i^{(k)}$ are defined by \eqref{his}. Thanks to $\theta_i^{(k)}>0$, $\lambda_i>0$ and the monotonicity of $e^{\tau\lambda_i\Delta t/T }$ with respect to $\tau$, $A_i^{(k)}>0$, $B_i^{(k)}>0$ hold and then
\begin{align*}
0<\rho_{k}^{(k)}<\rho_{k-1}^{(k)}<\rho_{k-2}^{(k)}<\ldots<\rho_1^{(k)}.
\end{align*}
According to $\prescript{}{}{g}_k^{(k)}\geq \frac{(1-\alpha_{k+\sigma_k})}{2(k+\sigma_k)^{\alpha_{k+\sigma_k}}}$ \cite{Du-2020}, \eqref{L21},\eqref{soe}, and \eqref{coe-FL21}, we have
\begin{align*}
&\rho_k^{(k)}\geq (1-\epsilon)\prescript{}{}{g}_k^{(k)}\geq \frac{(1-\epsilon)(1-\alpha_{k+\sigma_k})} {2 (k+\sigma_k)^{\alpha_{k+\sigma_k}}}.
\end{align*}

So condition \eqref{coe1} will hold if $\rho_0^{(k)}>\rho_1^{(k)}$ holds, an inequality which obviously follows from the condition \eqref{coe2}. Thus it is enough to prove the later one. By \eqref{coe}, the left-hand side of the condition \eqref{coe2} satisfies
\begin{align}\label{coe6}
&(2\sigma_k-1)\rho_0^{(k)}-\sigma_k\rho_1^{(k)}\nonumber\\
\geq &(2\sigma_k-1)\prescript{}{}{g}_0^{(k)}-\sigma_k\prescript{}{}{g}_1^{(k)}-(2\sigma_k-1)\frac {(1-\alpha_{k+\sigma_k})\Delta t^{-\alpha_{k+\sigma_k}}\epsilon }{4}
      -(1-\alpha_{k+\sigma_k})\Delta t^{-\alpha_{k+\sigma_k}}\epsilon
\end{align}
for $k=1$, and
\begin{align}\label{coe7}
&(2\sigma_k-1)\rho_0^{(k)}-\sigma_k\rho_1^{(k)}\nonumber\\
\geq & (2\sigma_k-1)\prescript{}{}{g}_0^{(k)}-\sigma_k\prescript{}{}{g}_1^{(k)}-(2\sigma_k-1)\frac {(1-\alpha_{k+\sigma_k})\Delta t^{-\alpha_{k+\sigma_k}}\epsilon }{4}
       -\sigma_k\frac {5(1-\alpha_{k+\sigma_k})\Delta t^{-\alpha_{k+\sigma_k}}\epsilon }{4}
\end{align}
for $ k\geq2$.

From \eqref{L21}, the corresponding property about $\prescript{ }{}{g}_l^{(k)}$ is
\begin{align*}
(2\sigma_k-1)\prescript{ }{}{g}_0^{(k)}-\sigma_k\prescript{ }{}{g}_1^{(k)}
%=&\left\{
%\begin{array}{ll}
%(2\sigma_k-1)(a_0^{(k)}+b_1^{(k)})-\sigma_k(a_1^{(k)}-b_1^{(k)}), \ \ \ \ \ \ \ \ \ \ \ k=1,\\
%(2\sigma_k-1)(a_0^{(k)}+b_1^{(k)})-\sigma_k(a_1^{(k)}+b_2^{(k)}-b_1^{(k)}), \ \ \ k\geq2.
%  \end{array}
%\right.\\
=&\left\{
\begin{array}{ll}
\frac{(2\sigma_k-1)(1-\sigma_k)}{2\sigma_k}(1+\sigma_k)^{1-\alpha_{k+\sigma_k}},\ \ k=1,\\
(1+\sigma_k)^{1-\alpha_{k+\sigma_k}}\Big(\frac{4\sigma_k-1}{2\sigma_k}-\big(\frac{2+\sigma_k}{1+\sigma_k}\big)^{1-\alpha_{k+\sigma_k}}\Big),\ \ k\geq2.
  \end{array}
\right.\\
\geq&\left\{
\begin{array}{ll}
\frac{(2\sigma_k-1)(1-\sigma_k)}{2\sigma_k(1+\sigma_k)^{\alpha_{k+\sigma_k}-1}},\ \ k=1,\\
\frac{(2\sigma_k-1)(1-\sigma_k)}{2\sigma_k(1+\sigma_k)^{\alpha_{k+\sigma_k}}},\ \ k\geq2,
  \end{array}
\right.\\
>&0.
\end{align*}
Thus in order to make the condition \eqref{coe2} hold we just need
\begin{align*}
\frac{(2\sigma_k-1)(1-\sigma_k)}{2\sigma_k(1+\sigma_k)^{\alpha_{k+\sigma_k}-1}}-(2\sigma_k-1)\frac {(1-\alpha_{k+\sigma_k})\Delta t^{-\alpha_{k+\sigma_k}}\epsilon }{4}-(1-\alpha_{k+\sigma_k})\Delta t^{-\alpha_{k+\sigma_k}}\epsilon \geq0,\ \ k=1,
\end{align*}
and
\begin{align*}
\frac{(2\sigma_k-1)(1-\sigma_k)}{2\sigma_k(1+\sigma_k)^{\alpha_{k+\sigma_k}}}-(2\sigma_k-1)\frac {(1-\alpha_{k+\sigma_k})\Delta t^{-\alpha_{k+\sigma_k}}\epsilon }{4}-\sigma_k\frac {5(1-\alpha_{k+\sigma_k})\Delta t^{-\alpha_{k+\sigma_k}}\epsilon }{4}\geq0,\ \ k\geq2,
\end{align*}
that is
\begin{align*}
\epsilon\leq \frac{2(2\sigma_k-1)\Delta t^{\alpha_{k+\sigma_k}}}{\sigma_k(1+\sigma_k)^{\alpha_{k+\sigma_k}-1}} \min\bigg\{\frac 1{6\sigma_k-1}, \frac 1{7\sigma_k-1}\bigg\}.
\end{align*}
In $FL2$-$1_{\sigma}$ formula, $\sigma_k=1-\frac{\alpha_{k+\sigma_k}}2$. Thus we have
\begin{align*}
\epsilon\leq \frac{2(1-\alpha_{k+\sigma_k})\Delta t^{\alpha_{k+\sigma_k}}}{(6-\frac7 2\alpha_{k+\sigma_k})(1-\frac 1 2 \alpha_{k+\sigma_k})(2-\frac 1 2 \alpha_{k+\sigma_k})^{\alpha_{k+\sigma_k}-1}}.
\end{align*}
So we get \eqref{coe2} if this $\epsilon$ satisfies
\begin{align*}
\epsilon
\leq \frac{2(1-\overline{\alpha})(2-\frac 1 2 \overline{\alpha})^{1-\overline{\alpha}}\Delta t^{\overline{\alpha}}}{(6-\frac7 2\underline{\alpha})(1-\frac 1 2 \underline{\alpha})}
\leq \frac{2(1-\alpha_{k+\sigma_k})\Delta t^{\alpha_{k+\sigma_k}}}{(6-\frac7 2\alpha_{k+\sigma_k})(1-\frac 1 2 \alpha_{k+\sigma_k})(2-\frac 1 2 \alpha_{k+\sigma_k})^{\alpha_{k+\sigma_k}-1}}.
\end{align*}
The proof is complete.
\end{proof}

\section{Difference schemes for the sub-diffusion problem}\label{finite-difference-scheme}
In this section, based on $FL2$-$1_{\sigma}$ formula in Section \ref{fast-approximation}, we  construct a fast temporal second-order and spatial fourth-order finite difference method ($FL2$-$1_{\sigma}$ scheme) for the VO time-fractional sub-diffusion equations. The unconditional stability and convergence of the difference method are investigated.

 We consider the multi-dimensional VO time-fractional sub-diffusion equations
\begin{align}
&\prescript{C}{0}{\mathcal{D}}^{\alpha(t)}_tu(\mathbf{x},t)=\Delta u(\mathbf{x},t)+f(\mathbf{x},t),\ \ \mathbf{x}\in \Omega,\ \ t\in (0,T], \label{E1}\\
&u(\mathbf{x},0)=\varphi(\mathbf{x}),\ \ \mathbf{x}\in \overline{\Omega},\label{E2}\\
&u(\mathbf{x},t)=0,\ \ \mathbf{x}\in \partial\Omega,\ \ t\in [0,T],\label{E3}
\end{align}
where $\Omega=\prod_{k=1}^d(a_l^{(k)},a_r^{(k)})\subset \mathbb{R}^d$, $\partial\Omega$ is the boundary of $\Omega$, $\overline{\Omega}=\Omega\cup\partial\Omega$, $\mathbf{x}=(x^{(1)},x^{(2)},\ldots,x^{(d)})\in\Omega$, $\Delta u(\mathbf{x},t)=\sum_{k=1}^d\partial^2_{x^{(k)}}u(\mathbf{x},t)$, $f(\mathbf{x},t)$ and $\varphi(\mathbf{x})$ represent the given sufficiently smooth functions.

Denote $L^{(k)}=a_r^{(k)}-a_l^{(k)}$, $k=1,2,\ldots,d$. Let $m^{(k)}$ be positive integers. Define a uniform partition of $\Omega$ by $x^{(k)}_{j^{(k)}}=a_l^{(k)}+j^{(k)}\Delta x^{(k)}$ $(j^{(k)}=0,1,\ldots,m^{(k)}; k=1,2,\ldots,d)$ for $\Delta x^{(k)}=L^{(k)}/m^{(k)}$
and denote $\Delta x=\max\limits_{1\leq k \leq d}\Delta x^{(k)}$.
Let $\Omega_h=\{(x^{(1)}_{j^{(1)}},x^{(2)}_{j^{(2)}},\ldots,x^{(d)}_{j^{(d)}})|j^{(k)}=1,2,\ldots,m^{(k)}-1; k=1,2,\ldots,d\}$,  $\overline{\Omega}_h=\{(x^{(1)}_{j^{(1)}},x^{(2)}_{j^{(2)}},\ldots,x^{(d)}_{j^{(d)}})|j^{(k)}=0,1,\ldots,m^{(k)}; k=1,2,\ldots,d\}$,
and $\partial\Omega_h=\overline{\Omega}_h\setminus\Omega_h$.
Denote the index vector $j=(j^{(1)},j^{(2)},\ldots,j^{(d)})$ and the spatial point $\mathbf{x}_j=\big(x^{(1)}_{j^{(1)}},x^{(2)}_{j^{(2)}},\ldots,x^{(d)}_{j^{(d)}}\big)$, then we define the index space
\begin{align*}
\overline{\mathcal{J}}=\{j|\mathbf{x}_j\in\overline{\Omega}_h\},\ \ \ \ \mathcal{J}=\{j|\mathbf{x}_j\in\Omega_h\},\ \ \ \ \partial \mathcal{J}=\{j|\mathbf{x}_j\in\partial\Omega_h\}.
\end{align*}
Thus the grid function spaces are defined by
\begin{align*}
&\mathcal {U}=\{u|u\ \textrm{being a grid function on}\ \overline{\Omega}_h\},\\
&\mathring{\mathcal {U}}=\{u|u\in\mathcal {U}; u_j=0\ \textrm{when}\ j\in\partial \mathcal{J}\}.
\end{align*}
We introduce the following discrete operators in the grid space $\mathcal {U}$
\begin{align*}
\delta_ku_{j+\frac 1 2\delta_k}=\frac 1 {\Delta x^{(k)}}(u_{j+\delta_k}-u_j),\ \ \delta_k^2u_j=\frac 1 {(\Delta x^{(k)})^2}(u_{j+\delta_k}-2u_j+u_{j-\delta_k}),
\end{align*}
where $\delta_k=(0,\ldots,1,\ldots,0)$ is an index with 1 at the $k$-th position and 0 at other positions.
\begin{align}\label{cal-A}
&\mathcal{A}_ku_j=
\begin{cases}
\frac 1 {12}(u_{j-\delta_k}+10u_j+u_{j+\delta_k}), \ \ \ j\in \mathcal{J},\\
u_j, \ \ \ j\in\partial \mathcal{J},\\
\end{cases}\\
&\Delta_hu_j=\sum_{k=1}^d\delta_k^2u_j,\ \ \mathcal{A}_hu_j=\prod_{k=1}^d\mathcal{A}_ku_j, \ \ \Lambda_hu_j=\sum_{k=1}^d\prod_{l=1,l\neq k}^d\mathcal{A}_l\delta_k^2u_j.
\end{align}
In the grid function space $\mathring{\mathcal {U}}$, define the discrete inner products and norms
\begin{align*}
&(u,w)=\left(\prod_{r=1}^d\Delta x^{(r)}\right)\sum_{j\in \mathcal{J}}u_jw_j,\ \ \ \ \| u\| =\sqrt{(u,u)},\\
&(u,w)_{\mathcal{A}_h}=(\mathcal{A}_hu,w),\ \ \ \ \ \ \ \ \ \ \ \ \ \ \| u\| _{\mathcal{A}_h}=\sqrt{(u,{\mathcal{A}_h}u)},\\
&(\delta_ku,\delta_kw)
=\left(\prod_{r=1}^d\Delta x^{(r)}\right)\sum_{j^{(k)}=0}^{m^{(k)}-1}\left(\prod_{l=1,l\neq k}^d\sum_{j^{(l)}=1}^{m^{(l)}-1}\right)
 (\delta_ku_{j+\frac 1 2\delta_k})(\delta_kw_{j+\frac 1 2\delta_k}),\\
&|u|_{1,k}=\sqrt{(\delta_ku,\delta_ku)},\ \ \ \ \ \ \ \ \ \ \ \ \ \ |u|_1=\sqrt{\sum_{k=1}^d|u|_{1,k}^2},\\
&(\delta_k^2u,\delta_k^2w)=\left(\prod_{r=1}^d\Delta x^{(r)}\right)\sum_{j\in \mathcal{J}}(\delta_k^2u_j)(\delta_k^2w_j),\ \ |u|_{2,k}=\sqrt{(\delta_k^2u,\delta_k^2u)},\ \ \|u\|_\infty=\max_{j\in \mathcal{J}}|u_j|.\\
\end{align*}
\subsection{$FL2$-$1_{\sigma}$ scheme}
Before deriving the finite difference schemes for the problem \eqref{E1}--\eqref{E3}, we denote the numerical solution at time $t_k$ with spatial point $\mathbf{x}_j$ by $u_j^k$, and $f(\mathbf{x}_j,t_{k+\sigma_k})$ by $f_j^{k+\sigma_k}$ for $\mathbf{x}_j\in \overline{\Omega}_h$ and $0\leq k \leq n$. We assume that the solution $u\in \mathcal{C}^{(6,3)}\big(\overline{\Omega}\times [0,T]\big)$.

Next we recall $L2$-$1_{\sigma}$ scheme for the problem \eqref{E1}--\eqref{E3}. Considering the equation \eqref{E1} at $(\mathbf{x}_j,t_{k+\sigma_k})$, we have
\begin{align}\label{DE1}
&\prescript{C}{0}{\mathcal{D}}^{\alpha(t)}_tu(\mathbf{x}_j,t)|_{t=t_{k+\sigma_k}}=\Delta u(\mathbf{x}_j,t_{k+\sigma_k})+f_j^{k+\sigma_k},\ \ j\in \mathcal{J},\ \ 0\leq k \leq n-1.
\end{align}
By Lemma \ref{L-error-FL21}, the term on the left-hand side of \eqref{DE1} satisfies
\begin{align}\label{tru-error1}
&\prescript{C}{0}{\mathcal{D}}^{\alpha(t)}_tu(\mathbf{x}_j,t)|_{t=t_{k+\sigma_k}}
=\prescript{FH}{0}{\mathcal{D}}^{\alpha_{k+\sigma_k}}_tu(\mathbf{x}_j,t_{k+\sigma_k})
 +\mO(\Delta t^{3-\alpha_{k+\sigma_k}}+\epsilon),\ \ j\in \mathcal{J},\ \ 0\leq k \leq n-1.
\end{align}
On the other hand, using Lemmas 3.4 and 3.5 in \cite{Du-2020} , the first term on the right-hand side of \eqref{DE1} satisfies
\begin{align}\label{tru-error2}
\Delta u(\mathbf{x}_j,t_{k+\sigma_k})=&\sigma_k\Delta u(\mathbf{x}_j,t_{k+1})+(1-\sigma_k)\Delta u(\mathbf{x}_j,t_k)+\mO(\Delta t^2).
%=&\sigma_{k+\sigma_k}\Delta_h u_j^{k+1}+(1-\sigma_{k+\sigma_k})\Delta_h u_j^k+\mO(\Delta t^{3-\alpha_{k+\sigma_k}}+\Delta x^2)\nonumber\\
%=&\Delta_h(\sigma_{k+\sigma_k} u_j^{k+1}+(1-\sigma_{k+\sigma_k}) u_j^k)+\mO(\Delta t^{3-\alpha_{k+\sigma_k}}+\Delta x^2),\ \ \ i\in I,\ \ 0\leq k \leq n-1.
\end{align}
Then substituting \eqref{tru-error1} and \eqref{tru-error2} into \eqref{DE1} gives
\begin{align*}
\prescript{FH}{0}{\mathcal{D}}^{\alpha_{k+\sigma_k}}_tu(\mathbf{x}_j,t_{k+\sigma_k})
=\sigma_k \Delta u(\mathbf{x}_j,t_{k+1})
 +(1-\sigma_k) \Delta u(\mathbf{x}_j,t_k)
 + f_j^{k+\sigma_k}+\mO(\Delta t^2+\epsilon),\\
\ \ \ \ \ \ \ \ \ \ \ \ \ \ \ \ \ \ \ \ \ \ \ j\in \overline{\mathcal{J}},\ \ 0\leq k \leq n-1.
\end{align*}
Acting the averaging operator $\mathcal{A}_h$ in \eqref{cal-A} on both hand sides of the equality above, we obtain
\begin{align*}
&\mathcal{A}_h\prescript{FH}{0}{\mathcal{D}}^{\alpha_{k+\sigma_k}}_tu(\mathbf{x}_j,t_{k+\sigma_k})\\
=&\sigma_k\mathcal{A}_h\Delta u(\mathbf{x}_j,t_{k+1})
 +(1-\sigma_k)\mathcal{A}_h\Delta u(\mathbf{x}_j,t_k)
 +\mathcal{A}_hf_j^{k+\sigma_k}+\mO(\Delta t^2+\epsilon),\ \ j\in \mathcal{J},\ \ 0\leq k \leq n-1,
\end{align*}
where
\begin{align*}
\prescript{FH}{0}{\mathcal{D}}^{\alpha_{k+\sigma_k}}_tu(\mathbf{x}_j,t_{k+\sigma_k})
=\frac{T^{-\alpha_{k+\sigma_k}}}{\Gamma(1-\alpha_{k+\sigma_k})}\sum_{i=\underline{N}+1}^{\overline{N}}\theta_{i}^{(k)}H_i^{(k)}
 + s^{(k)}\sigma_k^{1-\alpha_{k+\sigma_k}}\big(u(\mathbf{x}_j,t_{k+1})-u(\mathbf{x}_j,t_k)\big),
\end{align*}
with
\begin{align*}
H_i^{(k)}=&e^{-\lambda_i(1+\sigma_k-\sigma_{k-1})\Delta t/T} H_i^{(k-1)}
           +A_i^{(k)}\big(u(\mathbf{x}_j,t_k)-u(\mathbf{x}_j,t_{k-1})\big)+B_i^{(k)}\big(u(\mathbf{x}_j,t_{k+1})-u(\mathbf{x}_j,t_k)\big).
\end{align*}
By the fact of \cite{Du-2020}
\begin{align*}
\mathcal{A}_h\Delta u(\mathbf{x}_j,t_k)=\Lambda_hu(\mathbf{x}_j,t_k)+\mO(\Delta x^4),\ \ j\in \mathcal{J},\ \ 0\leq k \leq n,
\end{align*}
thus
\begin{align}\label{CADIFL21}
\mathcal{A}_h\prescript{FH}{0}{\mathcal{D}}^{\alpha_{k+\sigma_k}}_tu(\mathbf{x}_j,t_{k+\sigma_k})
=\Lambda_h\big(\sigma_k u(\mathbf{x}_j,t_{k+1})
 +(1-\sigma_k) u(\mathbf{x}_j,t_k)\big)
 +\mathcal{A}_hf_j^{k+\sigma_k}+S_j^k,\nonumber\\
 ~~~~~~~~~~~~~~~~~~~~~~~~~~~~~~~~~~~~~\ \ j\in \mathcal{J},\ \ 0\leq k \leq n-1,
\end{align}
where there exists a constant $c_0$ such that
\begin{align}\label{local truncation}
|S_j^k|\leq c_0 (\Delta t^2+\Delta x^4+\epsilon),\ \ \ \ j\in \mathcal{J},\ \ 0\leq k \leq n-1.
\end{align}
From the initial and boundary value conditions \eqref{E2}--\eqref{E3}, we have
\begin{align}
&u(\mathbf{x}_j,0)=\varphi(\mathbf{x}_j),\ \ j\in \mathcal{J},\label{IBV1}\\
&u(\mathbf{x}_j,t_k)=0,\ \ j\in \partial\mathcal{J},\ \ 0\leq k \leq n.\label{IBV2}
\end{align}
Omitting the small term $S_j^k$ in \eqref{CADIFL21}, we construct $FL2$-$1_{\sigma}$ scheme for the problem \eqref{E1}--\eqref{E3} as follows
\begin{align}
&\mathcal{A}_h\prescript{FH}{0}{\mathcal{D}}^{\alpha_{k+\sigma_k}}_tu_j^{k+\sigma_k}
=\Lambda_h\left(\sigma_k u_j^{k+1}+(1-\sigma_k) u_j^k\right)+\mathcal{A}_hf_j^{k+\sigma_k},\ \ j\in \mathcal{J},\ \ 0\leq k \leq n-1,\label{FL211}\\
&u_j^0=\varphi(\mathbf{x}_j),\ \ j\in \mathcal{J},\label{FL212}\\
&u_j^k=0,\ \ j\in \partial\mathcal{J},\ \ 0\leq k \leq n,\label{FL213}
\end{align}
where
\begin{align*}
\prescript{FH}{0}{\mathcal{D}}^{\alpha_{k+\sigma_k}}_tu_j^{k+\sigma_k}
=\frac{T^{-\alpha_{k+\sigma_k}}}{\Gamma(1-\alpha_{k+\sigma_k})}\sum_{i=\underline{N}+1}^{\overline{N}}\theta_{i}^{(k)}H_i^{(k)}
 + s^{(k)}\sigma_k^{1-\alpha_{k+\sigma_k}}\big(u_j^{k+1}-u_j^k\big),
\end{align*}
with
\begin{align*}
H_i^{(k)}=&e^{-\lambda_i(1+\sigma_k-\sigma_{k-1})\Delta t/T} H_i^{(k-1)}+A_i^{(k)}\big(u_j^k-u_j^{k-1}\big)+B_i^{(k)}\big(u_j^{k+1}-u_j^k\big).
\end{align*}
Recall $L2$-$1_{\sigma}$ scheme for the problem \eqref{E1}--\eqref{E3} as follows \cite{Du-2020}
\begin{align}
&\mathcal{A}_h\prescript{H}{0}{\mathcal{D}}^{\alpha_{k+\sigma_k}}_tu_j^{k+\sigma_k}
=\Lambda_h\left(\sigma_k u_j^{k+1}+(1-\sigma_k) u_j^k\right)+\mathcal{A}_hf_j^{k+\sigma_k},\ \ j\in \mathcal{J},\ \ 0\leq k \leq n-1,\label{L211}\\
&u_j^0=\varphi(\mathbf{x}_j),\ \ j\in \mathcal{J},\label{L212}\\
&u_j^k=0,\ \ j\in \partial\mathcal{J},\ \ 0\leq k \leq n,\label{L213}
\end{align}
where
\begin{align*}
\prescript{H}{0}{\mathcal{D}}^{\alpha_{k+\sigma_k}}_tu_j^{k+\sigma_k}
=&s^{(k)}\sum_{l=0}^k g_l^{(k)}\big(u_j^{k-l+1}-u_j^{k-l}\big).
\end{align*}
\subsection{Stability and convergence of  $FL2$-$1_{\sigma}$ scheme}

As described in Lemma \ref{L-coefficient}, the coefficients $\rho_l^{(k)}$ hold the vital properties for the stability and convergence analysis. Thus, similar to the proof given in \cite{Du-2020}, we present the following lemmas which will be used in the analysis of $FL2$-$1_{\sigma}$ scheme \eqref{FL211}--\eqref{FL213}.

\begin{lemma}\cite{Alikhanov-2015}\label{L-3.3}
Let $\mathring{\mathcal {U}}$ be an inner product space and $\langle\cdot,\cdot\rangle_*$ is the inner product with the induced norm $\|\cdot\|_*$. Suppose $\big\{c_l^{(k)}|0\leq l\leq k, k\geq1\big\}$ satisfies $0<c_k^{(k)}<c_{k-1}^{(k)}<\ldots<c_0^{(k)}$. For $v^0, v^1, \ldots, v^{k+1}\in \mathring{\mathcal {U}}$, we have the following inequality
\begin{align*}
\sum_{l=0}^kc_l^{(k)}\langle v^{k-l+1}-v^{k-l},\sigma_kv^{k+1}+(1-\sigma_k)v^k \rangle_*
\geq \frac 1 2 \sum_{l=0}^kc_l^{(k)}\Big(\|v^{k-l+1}\|_*^2-\|v^{k-l}\|_*^2\Big).
\end{align*}
\end{lemma}

\begin{lemma}\cite{Du-2020}\label{L-3.6}
For any $u, w \in \mathring{\mathcal {U}}$, define
\begin{align*}
\langle u, w \rangle_{\mathcal{A}_h}=\big(\mathcal{A}_hu,-\Delta_h w\big).
\end{align*}
Then $\langle u, w \rangle_{\mathcal{A}_h}$ is an inner product on $\mathring{\mathcal {U}}$. We denote
\begin{align*}
|u|_{1,\mathcal{A}_h}=\sqrt{\langle u, u \rangle_{\mathcal{A}_h}}.
\end{align*}
\end{lemma}

\begin{lemma}\cite{Du-2020}\label{L-3.8}
For any $u \in \mathring{\mathcal {U}}$, we have
\begin{align*}
\big(\frac 2 3\big)^d|u|_1^2\leq|u|_{1,\mathcal{A}_h}^2\leq|u|_1^2.
\end{align*}
\end{lemma}

\begin{lemma}\cite{Du-2020}\label{L-3.9}
For any $u \in \mathring{\mathcal {U}}$, we have
\begin{align*}
\big(\frac 2 3\big)^{d-1}\|\Delta_h u\|^2\leq\big(\Lambda_h u,\Delta_h u\big)\leq\|\Delta_h u\|^2.
\end{align*}
\end{lemma}

Now we obtain the following theorem to state the unconditional stability of the proposed scheme.
\begin{theorem}\label{T-stability}
Let $\big\{u_j^k|j\in \overline{\mathcal{J}}, 0\leq k\leq n\big\}$ be the solution of the difference scheme \eqref{FL211}--\eqref{FL213}. Then, we have
\begin{align}\label{stability}
\big|u^k\big|_{1,\mathcal{A}_h}^2
\leq \big|u^0\big|_{1,\mathcal{A}_h}^2+\big(\frac 3 2\big)^{d-1}\cdot\frac 1 {1-\epsilon} c_1
     \max_{0\leq l\leq k-1}\|\mathcal{A}_h f^{l+\sigma_l}\|^2,\ \ 0\leq k\leq n,
\end{align}
where
\begin{align*}
\|\mathcal{A}_h f^{l+\sigma_l}\|^2
=\left(\prod_{r=1}^d\Delta x^{(r)}\right)\sum_{j\in \mathcal{J}}\left(\mathcal{A}_h f_j^{l+\sigma_l}\right)^2,\ \
 c_1=\max_{0\leq t\leq T}\Big\{t^{\alpha(t)}\Gamma\big(1-\alpha(t)\big)\Big\}.
\end{align*}
\end{theorem}
\begin{proof}
Denote $u^{k+\sigma_k}=\sigma_ku^{k+1}+(1-\sigma_k)u^k$. Making an inner product with $-\Delta_h u^{k+\sigma_k}$ on both hand sides of \eqref{L211}, we obtain
\begin{align}\label{P1}
\big(\mathcal{A}_h\prescript{FH}{0}{\mathcal{D}}^{\alpha_{k+\sigma_k}}_tu^{k+\sigma_k},-\Delta_h u^{k+\sigma_k}\big)
+\big(\Lambda_hu^{k+\sigma_k},\Delta_h u^{k+\sigma_k}\big)
=-\big(\mathcal{A}_hf^{k+\sigma_k},\Delta_h u^{k+\sigma_k}\big),\nonumber\\
~~~~~~~~~~~~~~~~~~~~~~~~~~~~~~~~~~~~~~~~~0\leq k \leq n-1.
\end{align}
Noticing \eqref{FL213}, by Lemmas \ref{L-3.3} and \ref{L-3.6}, we get
\begin{align*}
\sum_{l=0}^k\rho_l^{(k)}\langle u^{k-l+1}-u^{k-l}, \sigma_ku^{k+1}+(1-\sigma_k)u^k \rangle_{\mathcal{A}_h}\geq\frac 1 2 \sum_{l=0}^k\rho_l^{(k)} \Big(|u^{k-l+1}|_{1,\mathcal{A}_h}^2-|u^{k-l}|_{1,\mathcal{A}_h}^2\Big).
\end{align*}
By Lemma \ref{L-3.9}, the second term on the left-hand side of \eqref{P1} follows that
\begin{align*}
\big(\Lambda_hu^{k+\sigma_k},\Delta_h u^{k+\sigma_k}\big)\geq \big(\frac 2 3\big)^{d-1}\|\Delta_hu^{k+\sigma_k}\|^2.
\end{align*}
Substituting the two inequalities into \eqref{P1} and using the Cauchy-Schwarz inequality with the basic inequality $ab\leq \big(\frac 2 3\big)^{d-1}a^2+\big(\frac 3 2\big)^{d-1}\frac 1 4 b^2$, we have
\begin{align*}
&\frac {s^{(k)}} 2 \sum_{l=0}^k\rho_l^{(k)}\Big(|u^{k-l+1}|_{1,\mathcal{A}_h}^2-|u^{k-l}|_{1,\mathcal{A}_h}^2\Big)+\big(\frac 2 3\big)^{d-1}\|\Delta_hu^{k+\sigma_k}\|^2\\
\leq&-\big(\mathcal{A}_h f^{k+\sigma_k},\Delta u^{k+\sigma_k}\big)\leq\|\mathcal{A}_h f^{k+\sigma_k}\|\|\Delta u^{k+\sigma_k}\|\\
\leq&\big(\frac 2 3\big)^{d-1}\|\Delta u^{k+\sigma_k}\|^2+\big(\frac 3 2\big)^{d-1}\frac 1 4\|\mathcal{A}_h f^{k+\sigma_k}\|^2,\ \ 0\leq k \leq n-1,
\end{align*}
which follows that
\begin{align*}
\frac {s^{(k)}} 2 \sum_{l=0}^k\rho_l^{(k)}\Big(|u^{k-l+1}|_{1,\mathcal{A}_h}^2-|u^{k-l}|_{1,\mathcal{A}_h}^2\Big)\leq\big(\frac 3 2\big)^{d-1}\frac 1 4\|\mathcal{A}_h f^{k+\sigma_k}\|^2,\ \ 0\leq k \leq n-1.
\end{align*}
Further we get
\begin{align*}
\rho_0^{(k)}|u^{k+1}|_{1,\mathcal{A}_h}^2
\leq \sum_{l=0}^{k-1}\Big(\rho_l^{(k)}-\rho_{l+1}^{(k)}\Big)|u^{k-l}|_{1,\mathcal{A}_h}^2
     +\rho_k^{(k)}|u^0|_{1,\mathcal{A}_h}^2+\big(\frac 3 2\big)^{d-1}\frac 1 {2s^{(k)}}\|\mathcal{A}_h f^{k+\sigma_k}\|^2,\\
     ~~~~~~~~~~~~~~~~~~~~~~~~~~~~~~~~~~~~~~~~~~~~~~~~~~~~~~~~~~~~~~~~~ 0\leq k \leq n-1.
\end{align*}
According to Lemma \ref{L-coefficient}, it follows that
\begin{align*}
\frac 1 {2s^{(k)}\rho_k^{(k)}}
<\frac{1}{1-\epsilon} \frac{(k+\sigma_k)^{\alpha_k+\sigma_k}}{s^{(k)}(1-\alpha_{k+\sigma_k})}
=\frac{1}{1-\epsilon} t^{\alpha(t)}\Gamma\big(1-\alpha(t)\big)|_{t=t_{k+\sigma_k}}
\leq \frac{1}{1-\epsilon}c_1,
\end{align*}
which leads to
\begin{align}\label{M1}
\rho_0^{(k)}|u^{k+1}|_{1,\mathcal{A}_h}^2
\leq \sum_{l=0}^{k-1}\Big(\rho_l^{(k)}-\rho_{l+1}^{(k)}\Big)|u^{k-l}|_{1,\mathcal{A}_h}^2
     +\rho_k^{(k)}\Big(|u^0|_{1,\mathcal{A}_h}^2+\big(\frac 3 2\big)^{d-1}\frac 1 {1-\epsilon}c_1\|\mathcal{A}_h f^{k+\sigma_k}\|^2\Big),\nonumber\\
     ~~~~~~~~~~~~~~~~~~~~~~0\leq k \leq n-1.
\end{align}
Next the mathematical induction will be used to prove the inequality \eqref{stability}. It is valid for $k=0$. Assume \eqref{stability} is true for $0\leq k\leq q$, now we prove that \eqref{stability} is valid for $k=q+1$. From \eqref{M1}, we obtain
\begin{align*}
\rho_0^{{q}}|u^{q+1}|_{1,\mathcal{A}_h}^2
\leq &\sum_{l=0}^{q-1}\Big(\rho_l^{(q)}-\rho_{l+1}^{(q)}\Big)|u^{q-l}|_{1,\mathcal{A}_h}^2
      +\rho_q^{(q)}\Big(|u^0|_{1,\mathcal{A}_h}^2+\big(\frac 3 2\big)^{d-1}\frac 1 {1-\epsilon}c_1\|\mathcal{A}_h f^{q+\sigma_q}\|^2\Big)\\
\leq &\sum_{l=0}^{q-1}\Big(\rho_l^{(q)}-\rho_{l+1}^{(q)}\Big)\Big(|u^0|_{1,\mathcal{A}_h}^2
      +\big(\frac 3 2\big)^{d-1}\frac 1 {1-\epsilon}c_1\max_{0\leq s\leq q-l-1}\|\mathcal{A}_h f^{s+\sigma_s}\|^2\Big)\\
      &+\rho_q^{(q)}\Big(|u^0|_{1,\mathcal{A}_h}^2+\big(\frac 3 2\big)^{d-1}\frac 1 {1-\epsilon}c_1\|\mathcal{A}_h f^{q+\sigma_q}\|^2\Big)\\
\leq &\Big(\sum_{l=0}^{q-1}(\rho_l^{(q)}-\rho_{l+1}^{(q)})+\rho_q^{(q)}\Big)
      \Big(|u^0|_{1,\mathcal{A}_h}^2+\big(\frac 3 2\big)^{d-1}\frac 1 {1-\epsilon}c_1\max_{0\leq s\leq q}\|\mathcal{A}_h f^{s+\sigma_s}\|^2\Big)\\
\leq &\rho_0^{(q)}\Big(|u^0|_{1,\mathcal{A}_h}^2+\big(\frac 3 2\big)^{d-1}\frac 1 {1-\epsilon}c_1\max_{0\leq s\leq q}\|\mathcal{A}_h f^{s+\sigma_s}\|^2\Big).
\end{align*}
From the above inequality, we obtain
\begin{align*}
|u^{q+1}|_{1,\mathcal{A}_h}^2\leq |u^0|_{1,\mathcal{A}_h}^2+\big(\frac 3 2\big)^{d-1}\frac 1 {1-\epsilon}c_1\max_{0\leq s\leq q}\|\mathcal{A}_h f^{s+\sigma_s}\|^2.
\end{align*}
Therefore, inequality \eqref{stability} is valid for $k=q+1$. This completes the proof.
\end{proof}

Theorem \ref{T-stability} reveals the stability of the difference scheme \eqref{FL211}--\eqref{FL213} with respect to the initial value and the source term. The next theorem is about the convergence of $FL2$-$1_{\sigma}$ scheme.

\begin{theorem}
Let $u(\mathbf{x},t) \in \mathcal{C}^{(6,3)}\big(\overline{\Omega}\times [0,T]\big)$ be the exact solution of the problem \eqref{E1}--\eqref{E3}, and $\big\{u_j^k|j\in \overline{\mathcal{J}}, 0\leq k\leq n\big\}$ be the solution of the difference scheme \eqref{FL211}--\eqref{FL213}. Denote $e_j^k=u(\mathbf{x}_j,t_k)-u_j^k$. Then, we have
\begin{align*}
\big|e^k\big|_{1,\mathcal{A}_h}\leq
\sqrt{\big(\frac 3 2\big)^{d-1} \cdot\frac 1 {1-\epsilon} c_1 \prod_{r=1}^d L^{(r)}}c_0\big(\Delta t ^2+\Delta x^4+\epsilon\big),\ \ 0\leq k\leq n.
\end{align*}
Furthermore, we obtain
\begin{align*}
\big|e^k\big|_1
\leq \sqrt{\big(\frac 3 2\big)^{2d-1}\cdot\frac 1 {1-\epsilon} c_1\prod_{r=1}^d L^{(r)}}c_0\big(\Delta t ^2+\Delta x^4+\epsilon\big),\ \ 0\leq k\leq n.
\end{align*}
\end{theorem}
\begin{proof}
Subtracting \eqref{FL211}--\eqref{FL213} from \eqref{CADIFL21}, \eqref{IBV1}--\eqref{IBV2}, respectively, we obtain the error equations as follows
\begin{align*}
&\mathcal{A}_h\prescript{FH}{0}{\mathcal{D}}^{\alpha_{k+\sigma_k}}_te_j^k
=\Lambda_h\left(\sigma_k e_j^{k+1}+(1-\sigma_k) e_j^k\right)+S_j^k,\ \ j\in \mathcal{J},\ \ 0\leq k \leq n-1,\\
&e_j^0=0,\ \ j\in \mathcal{J},\\
&e_j^k=0,\ \ j\in \partial\mathcal{J},\ \ 0\leq k \leq n.
\end{align*}
Applying \eqref{local truncation} and the priori estimation \eqref{stability}, it leads to
\begin{align*}
\big|e^k\big|_{1,\mathcal{A}_h}^2
\leq &\big|e^0\big|_{1,\mathcal{A}_h}^2+\big(\frac 3 2\big)^{d-1}\cdot\frac 1 {1-\epsilon} c_1  \max_{0\leq l\leq k-1} \|S^l\|^2\\
\leq &\big(\frac 3 2\big)^{d-1}\cdot\frac 1 {1-\epsilon} c_1  \prod_{r=1}^d L^{(r)}\Big(c_0\big(\Delta t^2+\Delta x^4+\epsilon\big)\Big)^2,\ \ 0\leq k\leq n.
\end{align*}
Consequently, we have
\begin{align*}
\big|e^k\big|_{1,\mathcal{A}_h}\leq
\sqrt{\big(\frac 3 2\big)^{d-1}\cdot\frac 1 {1-\epsilon} c_1  \prod_{r=1}^d L^{(r)}}c_0\big(\Delta t^2+\Delta x^4+\epsilon\big) ,\ \ 0\leq k\leq n.
\end{align*}
Furthermore, from Lemma \ref{L-3.8}, we obtain
\begin{align*}
\big|e^k\big|_1
\leq\sqrt{\big(\frac 3 2\big)^d}\big|e^k\big|_{1,\mathcal{A}_h}
\leq\sqrt{\big(\frac 3 2\big)^{2d-1}\cdot\frac 1 {1-\epsilon} c_1\prod_{r=1}^d L^{(r)}}c_0\big(\Delta t ^2+\Delta x^4+\epsilon\big),\ \ 0\leq k\leq n.
\end{align*}
\end{proof}

\section{Numerical results} \label{numerical-results}
In this section, some  numerical examples are presented to verify the effectiveness of $FL2$-$1_{\sigma}$ scheme \eqref{FL211}--\eqref{FL213} compared with $L2$-$1_{\sigma}$ scheme \eqref{L211}--\eqref{L213}. All experiments are performed based on Matlab 2016b on a laptop with the configuration: Intel(R) Core(TM) i7-7500U CPU 2.70GHz and 8.00 GB RAM.

\begin{example}\label{example1}
We show the efficiency of $FL2$-$1_{\sigma}$ scheme \eqref{FL211}--\eqref{FL213} in 2D case comparing with the corresponding $L2$-$1_{\sigma}$ scheme \eqref{L211}--\eqref{L213}. Take $\Omega=(0,\pi)\times(0,\pi)$, $T=1$, and the exact solution is given as \cite{Du-2020}
\begin{align*}
u(\mathbf{x},t)=(t^3+3t^2+1)\sin x^{(1)}\sin x^{(2)}.
\end{align*}
Based on the exact solution, we calculate the initial value and the source term
\begin{align*}
&\varphi(\mathbf{x})=\sin x^{(1)}\sin x^{(2)},\\
&f(\mathbf{x},t)=\left(\frac{6}{\Gamma(4-\alpha(t))}t^{3-\alpha(t)}+\frac{6}{\Gamma(3-\alpha(t))}t^{2-\alpha(t)}+2(t^3+3t^2+1)\right)\sin x^{(1)}\sin x^{(2)}.
\end{align*}
\end{example}

We take $\Delta x^{(1)}=\Delta x^{(2)}=\Delta x$, $m^{(1)}=m^{(2)}=m$, denote
\begin{align*}
&E(\Delta x,\Delta t)=\big\|u^n-u(\mathbf{x},t_n)\big\|_\infty,\\
&Order_t=\log_2\big(E(\Delta x,2\Delta t)/E(\Delta x,\Delta t)\big),\\
&Order_{x,t}=\log_2\big(E(2\Delta x,2\Delta t)/E(\Delta x,\Delta t)\big).
\end{align*}

In the calculations, we set the expected accuracy $\epsilon=\Delta t^2$ and different spatial and temporal step sizes. Note that the linear systems arising from the two-dimensional problems provide coefficient matrices possessing block-tridiagonal-Toeplitz with tridiagonal-Toeplitz-blocks, which can be diagonalized by the discrete sine transforms \cite{Bini-1990, Serra-1999}. Therefore,  as a global after-processing optimization, a fast implementation of the inversion, in our numerical experiments,  is carried out by the fast sine transforms to reduce the computational cost.

Tables \ref{2Dt} and \ref{2Ds} list the results of Example \ref{example1}. Fine spatial size is fixed at $\Delta x=\pi/320$ in Table \ref{2Dt}. Both $L2$-$1_{\sigma}$ scheme and $FL2$-$1_{\sigma}$ scheme can achieve the second-order temporal accuracy. Table \ref{2Dt} also shows the developments of the CPU time and memory of the two schemes with respect to $n$. The CPU time of $FL2$-$1_{\sigma}$ scheme is increasing half as fast as that of $L2$-$1_{\sigma}$ scheme, which reveals the $\mO(n\log^2 n)$ and $\mO(n^2)$ computational complexity of the two schemes, respectively. We note that as $n$ goes up, the memory of $FL2$-$1_{\sigma}$ scheme grows slowly, while the storage requirement of $L2$-$1_{\sigma}$ scheme increases linearly with $n$. Especially, when $n=16000$, $FL2$-$1_{\sigma}$ scheme is finished in 40 min, while $FL2$-$1_{\sigma}$ scheme is failed due to excessive storage requirements. Table \ref{2Ds} gives the numerical results with $m$ varying from $20$ to $160$ and  $n=m^2$. The convergence rates satisfy the theoretical result in Section \ref{finite-difference-scheme}. Nevertheless, a significant reduction is reflected in computational cost and storage memory of $FL2$-$1_{\sigma}$ scheme on the refined mesh, compared with $L2$-$1_{\sigma}$ scheme.

%The errors and convergence orders, CPU time and storage of $L2$-$1_{\sigma}$ scheme and $FL2$-$1_{\sigma}$ scheme are showed in Table \ref{2Dt} with the spatial size fixed at $\Delta x=\pi/320$ and temporal step sizes refined from $\Delta t=1/2000$ to $1/16000$. Table \ref{2Dt} shows that both $FL2$-$1_{\sigma}$ and $L2$-$1_{\sigma}$ can achieve second-order temporal accuracy, while the CPU time and the memory in workspace of $FL2$-$1_{\sigma}$ scheme are both pretty small, compared with $L2$-$1_{\sigma}$ scheme.
%
%
%the developments of CPU time and memory of the two schemes with respect to $n$. The CPU time and the memory of the $L1$ scheme increase much faster than those of the fast ESA scheme
%
%while in Table \ref{2Ds} with the spatial step size refined from $\Delta x=\pi/20$ to $\pi/160$ and temporal step sizes $\Delta t=\Delta x^2$.
%
%Tables \ref{2Dt} and \ref{2Ds} show that both $FL2$-$1_{\sigma}$ and $L2$-$1_{\sigma}$ can achieve second-order temporal and fourth-order spatial accuracy. From Table \ref{2Dt}, we conclude that
%Nevertheless, compared with $L2$-$1_{\sigma}$ scheme, the CPU time and the memory in workspace of $FL2$-$1_{\sigma}$ scheme are both pretty small.

\begin{table}[t]
\begin{center}
\caption{Convergence rates and the CPU time, memory of $L2$-$1_{\sigma}$ scheme and $FL2$-$1_{\sigma}$ scheme for Example \ref{example1} with $\alpha (t)=(2+\sin t)/4$, $m=320$, $\epsilon=\Delta t^2$.}
\label{2Dt}
\def\temptablewidth{1.0\textwidth}
{\rule{\temptablewidth}{1.0pt}}
{\footnotesize
\begin{tabular*}{\temptablewidth}{@{\extracolsep{\fill}}cccccccccccc}
& &\multicolumn{4}{c}{{$L2$-$1_{\sigma}$ scheme}} & \multicolumn{4}{c}{{$FL2$-$1_{\sigma}$ scheme}} \\
 \cline{3-6}\cline{7-10}
&$n$    &$E(\Delta x,\Delta t)$      &$Order_t$   &CPU(s)     &Memory      &$E(\Delta x,\Delta t)$   &$Order_t$   &CPU(s)    &Memory  \\
\hline
&2000   &2.3592e-7                   & -          &152.77     &1.64e+9     &2.3497e-7                &-           &194.70    &1.01e+8  \\
&4000   &5.8588e-8                   &2.01        &642.33     &3.27e+9     &5.8411e-8                &2.01        &470.72    &1.17e+8  \\
&8000   &1.4339e-8                   &2.03        &2510.36    &6.52e+9     &1.4319e-8                &2.03        &1021.68   &1.34e+8  \\
&16000  &-                           & -          &-          &-           &3.3034e-9                &2.12        &2311.46   &1.53e+8  \\
\end{tabular*} {\rule{\temptablewidth}{1pt}}
}
\end{center}
\end{table}

\begin{table}[t]
\begin{center}
\caption{Convergence rates and the CPU time, memory of $L2$-$1_{\sigma}$ scheme and $FL2$-$1_{\sigma}$ scheme for Example \ref{example1} with $\alpha (t)=(2+\sin t)/4$, $n=m^2$, $\epsilon=\Delta t^2$.}
\label{2Ds}
\def\temptablewidth{1.0\textwidth}
{\rule{\temptablewidth}{1.0pt}}
{\footnotesize
\begin{tabular*}{\temptablewidth}{@{\extracolsep{\fill}}cccccccccccccc}
& &\multicolumn{4}{c}{{$L2$-$1_{\sigma}$ scheme}} & \multicolumn{4}{c}{{$FL2$-$1_{\sigma}$ scheme}} \\
 \cline{3-6}\cline{7-10}
&$m$  &$E(\Delta x,\Delta t)$   &$Order_{x,t}$  &CPU(s)     &Memory      &$E(\Delta x,\Delta t)$   &$Order_{x,t}$   &CPU(s)   &Memory  \\
\hline
&20   &1.1392e-6                & -             &0.15       &1.20e+6     &1.1971e-6                &-               &0.14     &2.56e+5 \\
&40   &7.2797e-8                &3.97           &3.23       &1.96e+7     &7.4374e-8                &4.01            &2.05     &1.48e+6  \\
&80   &4.6192e-9                &3.98           &185.32     &3.20e+8     &4.6405e-9                &4.00            &52.46    &7.95e+6  \\
&160  &2.4931e-10               &4.21           &7158.69    &5.18e+9     &2.4589e-10               &4.24            &1073.38  &4.15e+7  \\
\end{tabular*} {\rule{\temptablewidth}{1pt}}
}
\end{center}
\end{table}

\begin{example}\label{example2}
We also show the efficiency of $FL2$-$1_{\sigma}$ scheme \eqref{FL211}--\eqref{FL213} in 3D case. Take $\Omega=(0,\pi)\times(0,\pi)\times(0,\pi)$, $T=1$, and the exact solution is given as
\begin{align*}
u(\mathbf{x},t)=(t^3+3t^2+1)\sin x^{(1)}\sin x^{(2)}\sin x^{(3)}.
\end{align*}
Based on the exact solution, we calculate the initial value and the source term
\begin{align*}
&\varphi(\mathbf{x})=\sin x^{(1)}\sin x^{(2)}\sin x^{(3)},\\
&f(\mathbf{x},t)=\left(\frac{6}{\Gamma(4-\alpha(t))}t^{3-\alpha(t)}+\frac{6}{\Gamma(3-\alpha(t))}t^{2-\alpha(t)}+3(t^3+3t^2+1)\right)\sin x^{(1)}\sin x^{(2)}\sin x^{(3)}.
\end{align*}
\end{example}

We take $\Delta x^{(1)}=\Delta x^{(2)}=\Delta x^{(3)}=\Delta x$, $m^{(1)}=m^{(2)}=m^{(3)}=m$, and set the expected accuracy $\epsilon=\Delta t^2$. The errors and convergence orders, CPU time and storage of $L2$-$1_{\sigma}$ scheme and $FL2$-$1_{\sigma}$ scheme are showed in Table \ref{3Dt} with $m=100$ and temporal step sizes refined from $\Delta t=1/200$ to $\Delta t=1/1600$, while in Table \ref{3Ds} $m$ varies from $10$ to $80$ and $n=(2m)^2$. The fast sine transforms is used to reduce the computational cost.

Tables \ref{3Dt} and \ref{3Ds} show that both $FL2$-$1_{\sigma}$ scheme and $L2$-$1_{\sigma}$ scheme can achieve the optimal convergence for three-dimensional problems. Nevertheless, complexity of $FL2$-$1_{\sigma}$ scheme is much cheaper than that of $L2$-$1_{\sigma}$ scheme on the refined mesh. For the fine mesh, however, $L2$-$1_{\sigma}$ scheme has very low efficiency or even cannot proceed because of the limit of memory, while $FL2$-$1_{\sigma}$ scheme solves the relevant problem very well.

\begin{table}[t]
\begin{center}
\caption{Convergence rates and the CPU time, memory of $L2$-$1_{\sigma}$ scheme and $FL2$-$1_{\sigma}$ scheme for Example \ref{example2} with $\alpha (t)=(2+\sin t)/4$, $m=100$, $\epsilon=\Delta t^2$.}
\label{3Dt}
\def\temptablewidth{1.0\textwidth}
{\rule{\temptablewidth}{1.0pt}}
{\footnotesize
\begin{tabular*}{\temptablewidth}{@{\extracolsep{\fill}}cccccccccccccc}
& &\multicolumn{4}{c}{{$L2$-$1_{\sigma}$ scheme}} & \multicolumn{4}{c}{{$FL2$-$1_{\sigma}$ scheme}} \\
 \cline{3-6}\cline{7-10}
&$n$   &$E(\Delta x,\Delta t)$  &$Order_t$  &CPU(s)    &Memory     &$E(\Delta x,\Delta t)$    &$Order_t$  &CPU(s)   &Memory  \\
\hline
&200   &2.6755e-5               & -         &41.64     &1.64e+9    &2.6587e-5                 &-          &99.77    &5.59e+8  \\
&400   &6.6637e-6               &2.01       &156.85    &3.19e+9    &6.6318e-6                 &2.00       &233.83   &6.6.e+8  \\
&800   &1.6528e-6               &2.01       &680.60    &6.30e+9    &1.6475e-6                 &2.01       &564.19   &7.84e+8  \\
&1600  &-                       & -         &-         &-          &4.0177e-7                 &2.04       &1358.80  &9.24e+8  \\
\end{tabular*} {\rule{\temptablewidth}{1pt}}
}
\end{center}
\end{table}

\begin{table}[t]
\begin{center}
\caption{Convergence rates and the CPU time, memory of $L2$-$1_{\sigma}$ scheme and $FL2$-$1_{\sigma}$ scheme for Example \ref{example2}  with $\alpha (t)=(2+\sin t)/4$,  $n=(2m)^2$, $\epsilon=\Delta t^2$.}
\label{3Ds}
\def\temptablewidth{1.0\textwidth}
{\rule{\temptablewidth}{1.0pt}}
{\footnotesize
\begin{tabular*}{\temptablewidth}{@{\extracolsep{\fill}}cccccccccccccc}
& &\multicolumn{4}{c}{{$L2$-$1_{\sigma}$ scheme}} & \multicolumn{4}{c}{{$FL2$-$1_{\sigma}$ scheme}} \\
 \cline{3-6}\cline{7-10}
&$m$ &$E(\Delta x,\Delta t)$  &$Order_{x,t}$  &CPU(s)    &Memory      &$E(\Delta x,\Delta t)$   &$Order_{x,t}$   &CPU(s)    &Memory  \\
\hline
&$10$  &1.2679e-4               & -             &0.32      &2.41e+6     &1.2682e-4                &-               &0.52      &5.06e+5 \\
&$20$  &7.9012e-6               &4.00           &13.54     &8.84e+7     &7.9021e-6                &4.00            &12.61     &6.56e+6  \\
&$40$  &4.9351e-7               &4.00           &1047.68   &3.04e+9     &4.9351e-7                &4.00            &508.26    &7.46e+7  \\
&$80$  & -                      &-              &-         &-           &3.0832e-8                &4.00            &20550.89  &8.01e+8  \\
\end{tabular*} {\rule{\temptablewidth}{1pt}}
}
\end{center}
\end{table}

\section{Concluding Remarks} \label{concluding-remarks}
In this paper, we consider the fast high-order evaluation for the VO Caputo fractional derivative. $FL2$-$1_\sigma$ formula can efficiently reduce the computational storage and cost for the VO fractional derivative. The fast formula is applied to construct $FL2$-$1_\sigma$ scheme for the multi-dimensional VO fractional sub-diffusion equations. We show the properties of $FL2$-$1_{\sigma}$ scheme to study the stability and convergence. Numerical examples are given to verify the theoretical results.

To our best knowledge, there is a shortage of effective ways to solve the VO fractional problem with a weak singularity at the initial time. We refer to \cite{Liao-2018, Stynes-2017} for the use of the nonuniform mesh sizes, which is a valid tool for the CO problems. Therefore, it is consequential to develop fast algorithms for the VO fractional problem with a weak singularity on nonuniform girds. We will consider the case in the future.

\section*{Acknowledgement}
The authors would like to thank Professor
Zhi-zhong Sun for his helpful suggestions during the preparation of this paper.

\end{document}